\documentclass[11pt,letterpaper]{article}

\usepackage[affil-it]{authblk}

\usepackage[utf8]{inputenc}
\usepackage{amsmath}
\usepackage{amsfonts}
\usepackage{amssymb}
\usepackage{amsthm}
\usepackage{graphicx}
\usepackage{subcaption} 
\usepackage[left=2cm,right=2cm,top=2cm,bottom=2cm]{geometry}
\usepackage{enumitem}
\usepackage{appendix}
\usepackage{mathtools}
\usepackage{xcolor}
\usepackage[colorlinks=true,linkcolor=blue,citecolor=magenta]{hyperref}
\usepackage{todonotes}
\usepackage[english]{babel}
\usepackage{tikz-cd}

\newtheorem{theorem}{Theorem}[section]
\newtheorem{lemma}[theorem]{Lemma}
\newtheorem{prop}[theorem]{Proposition}
\newtheorem{corollary}[theorem]{Corollary}
\theoremstyle{definition}

\newtheorem{assumption}[theorem]{Assumption}

\theoremstyle{remark}
\newtheorem{remark}[theorem]{Remark}

\newcommand{\pr}{\textnormal{pr}}
\newcommand{\PP}{\textnormal{P}}
\newcommand{\QQ}{\textnormal{Q}}

\numberwithin{equation}{section}

\begin{document}
\title{\bf Approximate Slow Manifolds \\ in the Fokker-Planck Equation} 
\date{}
\author{Christian Kuehn$\,^1$, Jan-Eric Sulzbach$\,^{2,*}$}
\affil{\small
    $\,^1$ Department of Mathematics, Technical University of Munich\\
    email: ckuehn@ma.tum.de\\
    
    \vspace{0.3cm}
    $\,^2$ Department of Mathematics, Technical University of Munich\\
    email: janeric.sulzbach@tum.de\\
    $\,*$ corresponding author\\
    
}

\maketitle

\begin{abstract}
   In this paper we study the dynamics of a fast-slow Fokker-Planck partial differential equation (PDE) viewed as the evolution equation for the density of a multiscale planar stochastic differential equation (SDE).
   Our key focus is on the existence of a slow manifold on the PDE level, which is a crucial tool from the geometric singular perturbation theory allowing the reduction of the system to a lower dimensional slowly evolving equation. In particular, we use a projection approach based upon a Sturm-Liouville eigenbasis to convert the Fokker-Planck PDE to an infinite system of PDEs that can be truncated/approximated to any order. Based upon this truncation, we can employ the recently developed theory for geometric singular perturbation theory for slow manifolds for infinite-dimensional evolution equations. This strategy presents a new perspective on the dynamics of multiple time-scale SDEs as it combines ideas from several previously disjoint reduction methods.
\end{abstract}

\section{Introduction}\label{sec:introduction}




Complex dynamical systems with both, multiple time scales and uncertainty caused by noisy fluctuations, play an important role in both science and engineering.
Two particular fields of application are in climate science, where the systems usually exhibit a fast-slow and stochastic behaviour \cite{majda2001mathematical,majda2003systematic,monahan2011stochastic}, and in neuroscience \cite{lindner2004effects,laing2009stochastic,franovic2020dynamics}. 
An important tool that has been developed in the context of deterministic fast-slow systems is the geometric singular perturbation theory \cite{fenichel1971persistence,fenichel1979geometric,jones1995gspt,kuehn2015multiple}.
In this context invariant manifolds, which are geometric structures in the phase space that help describe the dynamical behavior of multiscale systems, play a major role. 
For dynamical systems with a fast-slow structure a slow manifold is a specific type of invariant manifold.
Broadly speaking, it is an invariant manifold represented as a graph over a suitable subset of phase space variables, called the slow variables. If the dynamics outside of the slow manifold is attracting in the remaining fast variables, the system can be reduced to a self-contained, lower-dimensional evolution equation that depends exclusively on the slow variables. This procedure can facilitate both, geometric and numerical analysis. Here we focus on the geometric viewpoint. A natural question is, whether the classically deterministic reduction techniques can be transferred to the stochastic setting.

A first answer has been given by Berglund and Gentz \cite{berglund2003geometric,berglund2006noise} who considered the following stochastic multiscale system
\begin{align}\label{eq: SDE intro}
\begin{split}
    \textnormal{d}x&= \frac{1}{\varepsilon} f(x,y,\varepsilon)\, \textnormal{d}t +\frac{\sigma}{\sqrt{\varepsilon}}F(x,y,\varepsilon)\,\textnormal{d}W_t,\quad x\in \mathbb{R},\\
    \textnormal{d}y&=  g(x,y,\varepsilon)\, \textnormal{d}t +\sigma' G(x,y,\varepsilon)\,\textnormal{d}W_t,\quad y\in \mathbb{R},
    \end{split}
\end{align}
where $W_t$ denotes a standard Brownian motion, $\sigma,\sigma'>0$ are positive parameters and $0<\varepsilon\ll 1$ is the small parameter encoding the multiscale behaviour of the system.
In their work the authors introduce the concept of an invariant manifold for the above multiscale stochastic differential system via a probabilistic point of view by working with the covariance as variable of interest.
Their goal is to estimate quantitatively the noise-induced spreading of typical paths as well as the probability of exceptional paths near a deterministic slow manifold. 
Moreover, they prove that the sample paths of the above system are concentrated in some neighborhood of the slow manifold of the corresponding deterministic system up to some time with certain probability.

The second answer was given by Schmalfuss and Schneider \cite{schmalfuss2008invariant} and later also by Duan et al \cite{duan2013slow,ren2015approximation}.
In these works the authors consider a random dynamical system with two time scales
\begin{align}\label{eq: RDS intro}
\begin{split}
     \frac{\textnormal{d} x}{\textnormal{d} t}&= \tilde f(\theta_t^1\omega, \theta^{2,\varepsilon}_t\omega,x,y),\quad x\in \mathbb{R},\\
     \frac{\textnormal{d} y}{\textnormal{d} t}&= \varepsilon \tilde g(\theta_t^1\omega, \theta^{2,\varepsilon}_t\omega,x,y),\quad y\in \mathbb{R},
\end{split}
\end{align}
where the stochastic noise is given by a metric dynamical system $\Theta=\big(\Omega,\mathcal{F},\mathbb{P},\theta \big)$.
For the definition of $\Theta$ and the relation between $f,g$ and $\tilde f,\tilde g$ we refer to \cite{schmalfuss1998random,schmalfuss2008invariant}.
The focus of these works is then to establish the existence of a random invariant manifold for the random dynamical system that eliminates the fast variables.
The approaches rely on an appropriate random graph transform and a Lyapunov-Perron operator method respectively.

A different method to think about multiscale stochastic dynamical systems is stemming from theoretical physics \cite{kaneko1981adiabatic,van1985elimination,gardiner2009stochastic}.
The idea is to transform the stochastic dynamical system \eqref{eq: SDE intro} into a Fokker-Planck equation of the form
\begin{align}\label{eq: FPE intro}
    \begin{split}
        \partial_t \rho^\varepsilon&= -\frac{1}{\varepsilon}\partial_x (f(x,y,\varepsilon)\rho^\varepsilon)+ \frac{\sigma^2}{2\varepsilon} \partial_{x,x}\big( F(x,y,\varepsilon)^2\rho^\varepsilon\big)- \partial_y (g(x,y,\varepsilon)\rho^\varepsilon)+ \frac{(\sigma')^2}{2}\partial_{y,y}\big(G(x,y,\varepsilon)^2\rho^\varepsilon\big),
    \end{split}
\end{align}
where $\rho^\varepsilon=\rho^\varepsilon(x,y,t;x_0,y_0,t_0)$ denotes the transition probability density from $\big(x(t_0),y(t_0)\big)=\big(x_0,y_0\big)$ to $\big(x(t),y(t)\big)$, and then eliminate the fast variables in this setting. 
This then yields a reduced Fokker-Planck equation corresponding to the evolution of the slow variables in the stochastic differential equation.
We want to point out that this method is closely related to the concept of stochastic averaging, see the results of \cite{lochak2012multiphase,hijawi1997nonlinear,pavliotis2008multiscale}.

In this work, we introduce a novel approach to the stochastic reduction method for the Fokker-Planck equation. 
Our starting point is the fast-slow stochastic differential equation \eqref{eq: SDE intro}, from which we derive the fast-slow Fokker-Planck equation \eqref{eq: FPE intro} under suitable boundary conditions. 
However, our objective extends beyond simply obtaining a reduced equation; we aim to explore the geometric structure of the reduced system within the framework of the geometric singular perturbation theory \cite{jones1995geometric,kuehn2015multiple}. 
The central result of this work is a generalization of Fenichel’s theorem \cite{fenichel1971persistence,fenichel1979geometric} in the context of fast-slow Fokker-Planck equations.

To be more precise, we first of all rewrite the fast-slow Fokker-Planck equation into a system of two equations by using a special projection operator $\PP$.
These projections appear also in the context of statistical mechanics and are used in the context of the Mori-Zwanzig formalism \cite{mori1965transport,zwanzig1960ensemble,grabert2006projection,j2007statistical,te2020projection}, where the formalism is used as a coarse-graining procedure to extract the physical relevant observables.
Hence, we obtain
\begin{align}\label{eq: intro split FPE}
    \partial_t v^\varepsilon&= \PP\bigg( \frac{(\sigma')^2}{2}\partial_{y,y}\big(G(x,y,\varepsilon)^2(v^\varepsilon+w^\varepsilon)\big)- \partial_y \big(g(x,y,\varepsilon)(v^\varepsilon+w^\varepsilon)\big)\bigg),\\
    \begin{split}
    \partial_t w^\varepsilon&=\frac{\sigma^2}{2\varepsilon} \partial_{x,x}\big( F(x,y,\varepsilon)^2 w^\varepsilon\big) -\frac{1}{\varepsilon}\partial_x (f(x,y,\varepsilon) w^\varepsilon)\\ \label{eq: intro split FPE end}
    &\quad + \QQ\bigg(\frac{(\sigma')^2}{2}\partial_{y,y}\big(G(x,y,\varepsilon)^2(v^\varepsilon+w^\varepsilon)\big) - \partial_y \big(g(x,y,\varepsilon)(v^\varepsilon+w^\varepsilon)\big)\bigg),
    \end{split}
\end{align}
where $\QQ=\textnormal{I}-\PP$, and $v^\varepsilon,w^\varepsilon$ are given by $\PP \rho^\varepsilon$ and $\QQ \rho^\varepsilon$ respectively.
Now, under suitable assumptions that also hold up in many applications we can take the formal limit as $\varepsilon\to 0$ and obtain
\begin{align}
    \partial_t v^0&= \PP\bigg( \frac{(\sigma')^2}{2}\partial_{y,y}\big(G(x,y,0)^2 v^0\big)- \partial_y \big(g(x,y,0)v^0 \big)\bigg)
\end{align}
defined on the critical manifold
\begin{align}
    S_0=\{ (w^0,v^0): w^0=0\}.
\end{align}
Yet, in this setting we cannot directly apply the slow manifold theory. 
To this end we further exploit the fast-slow structure by reducing equations \eqref{eq: intro split FPE}-\eqref{eq: intro split FPE end} via a spectral Galerkin reduction with parameter $j\in \mathbb{N}$, similar to the one presented in \cite{knezevic2009spectral}. This then concludes the stochastic reduction method, where we have converted the $2$d fast-slow Fokker-Planck equation \eqref{eq: FPE intro} into an infinite system of linear coupled fast-slow equations depending only on one spatial variable. In particular, we truncate the Galerkin approximation at an arbitrary finite order $J\in \mathbb{N}$ and call this system (FPEt) derived from \eqref{eq: intro split FPE}-\eqref{eq: intro split FPE end}. The system (FPEt) is still a very large system of PDEs and grows in system size if a better approximation is required. We develop for (FPEt) a generalization of Fenichel's Theory when the critical manifold is normally hyperbolic and attracting. In particular, we prove the existence of an invariant manifold for (FPEt). Hence, the main result can be stated as follows.

\begin{theorem}
    Under suitable boundary conditions, initial data and structural assumptions on the nonlinear functions arising in the Fokker-Planck equation, as stated in detail in the next section, there exists a slow manifold $S_\varepsilon$ for equation (FPEt) satisfying the following:
    \begin{itemize}
        \item The slow dynamics on the slow manifold $S_\varepsilon$ converges to the dynamics on the critical manifold of (FPEt) as $\varepsilon\rightarrow 0$;
        \item The slow manifold has a (Hausdorff semi-)distance of $\mathcal{O}(\varepsilon)$ to the critical manifold of (FPEt) as $\varepsilon\to 0$;
        \item The manifold $S_{\varepsilon}$ is a locally exponential attracting invariant manifold which has the same regularity as the critical manifold of (FPEt);
        \item  The semi-flow of the solution on the slow manifold $S_{\varepsilon}$ converges to the semi-flow on the critical manifold of (FPEt). 
    \end{itemize}
\end{theorem}

In particular, since the Galerkin truncation level can be chosen to any arbitrary finite order, we can then approximate trajectories of the original Fokker-Planck equation to any desired accuracy by  of (FPEt), so that we can view $S_\varepsilon$ as an approximate slow manifold for our original problem. 

\begin{corollary}
Under the same assumptions as in the previous theorem and given any fixed finite time $T>0$, there exists a Galerkin truncation level such that the dynamics on the slow manifold $S_\varepsilon$ approximates the original system \eqref{eq: intro split FPE}-\eqref{eq: intro split FPE end} up to order $\mathcal{O}(\varepsilon)$ for $t\in(0,T]$. 
\end{corollary}

The rest of this paper is organized as follows. We present the stochastic reduction method in Sections \ref{sec: stochastic reduction}  through \ref{sec: slow manifold in FPE}, which we divide into four parts.
In the first two steps we introduce the aforementioned projection operators $\PP$ and $\QQ$ and rewrite the Fokker-Planck equation into the fast-slow system \eqref{eq: intro split FPE}-\eqref{eq: intro split FPE end}.
In the third step we apply an eigenfunction decomposition to the functions $v^\varepsilon$ and $w^\varepsilon$ such that the coefficients appearing in the decomposition $\big(a_j\big)_{j\in \mathbb{N}}$ only depend on time $t$ and the slow spatial coordinate $y$. This allows us convert the Fokker-Planck equation to an infinite-dimensional, fast-slow system of coupled equations for the coefficients $a_j$. Section \ref{sec:abstract theory} serves as an intermezzo, where we develop the abstract theory for fast-slow partial differential equations with a fast diffusion term, based on a series of previous papers \cite{hummel2022slow,kuehn2023fast,kuehn2024infinite}. In the final step, Section \ref{sec: slow manifold in FPE}, we aim to apply this abstract theory to the fast-slow coefficient system at hand.
Here, we truncate the infinite system at a finite $J\in \mathbb{N}$.
Once we establish the existence of a slow manifold for the truncated system, we retrace the reduction process to obtain a corresponding slow manifold for the (FPEt), which is the primary contribution of this paper. In Section \ref{sec: examples}, we demonstrate the application of these results to a linear example. We conclude this paper in Section \ref{sec: discussion}, where we propose further extensions of this theory to more complex problems, such as fast-slow stochastic differential equations with L\'evy noise.

\section{Stochastic Reduction}\label{sec: stochastic reduction}

In this section we introduce the method of stochastic reduction as a starting point for the geometric understanding of a general multiscale Fokker-Planck equation (FPE).

\subsection{Problem Set-up}\label{sec: problem set up}

Let $0<\varepsilon\ll 1$ be a small parameter encoding the time-scale separation.
Then, we consider a general fast-slow stochastic differential equation in standard form 
\begin{align}\label{eq: fast-slow SDE}
\begin{split}
    \textnormal{d}x&= \frac{1}{\varepsilon} f(x,y) \textnormal{d}t +\frac{\sigma_1}{\sqrt{\varepsilon}}F(x,y)\textnormal{d}W_1,\quad x\in \mathbb{R},\\
    \textnormal{d}y&=  g(x,y) \textnormal{d}t +\sigma_2 G(x,y)\textnormal{d}W_2,\quad y\in \mathbb{R},
    \end{split}
\end{align}
where the time parameter $t$ runs over a time interval $t\in [0,T]$ for some time $T>0$. 
We assume that the non-linear drift terms
\begin{align*}
    f: \mathbb{R}\times  \mathbb{R}\to  \mathbb{R} ,\quad g: \mathbb{R}\times  \mathbb{R}\to  \mathbb{R}
\end{align*}
and the diffusion functions
\begin{align*}
    F:\mathbb{R}\times  \mathbb{R}\to  \mathbb{R} ,\quad G:\mathbb{R}\times  \mathbb{R}\to  \mathbb{R}
\end{align*}
are sufficiently regular.
Moreover, $W_i=W_i(t)$, for $i=1,2$ denote independent standard Brownian motions, i.e. $W(t)$ is a stochastic process with $W(0)=0$ and normally distributed independent increments $W(t)-W(s) \sim \mathcal{N}(0,t-s)$ for $0\leq s\leq t$.\\

From this system of stochastic differential equations we can derive the forward Kolmogorov or Fokker-Planck equation, for example via the Kramer-Moyal expansion, see \cite{van1992stochastic,risken1996fokker,gardiner2009stochastic} for the details.
The Fokker-Planck equation is then given by
\begin{align}\label{eq:fast-slow FPE}
    \begin{split}
        \partial_t \rho^\varepsilon(t,x,y)&= -\frac{1}{\varepsilon}\partial_x \big(f(x,y)\rho^\varepsilon(t,x,y)\big)+ \frac{\sigma^2}{2\varepsilon} \partial_{x,x}\big( F(x,y)^2\rho^\varepsilon(t,x,y)\big)\\
        &\quad - \partial_y \big(g(x,y)\rho^\varepsilon(t,x,y)\big)+ \frac{(\sigma')^2}{2}\partial_{y,y}\big(G(x,y)^2\rho^\varepsilon(t,x,y)\big),
    \end{split}
\end{align}
where the equation is posed on the infinite strip $\Omega=\mathbb{R}\times (-R,R)\subset \mathbb{R}^2$, for some $R>0$, 
i.e., in the $x$-direction or 'fast' direction, we consider the equation on the whole space and in the $y$-direction or 'slow' direction we consider the equation on a bounded domain $y\in (-R,R)$.
This equation can be written in the short form
\begin{align}\label{eq:FPE short form}
    \partial_t \rho^\varepsilon= \big(\frac{1}{\varepsilon}\mathfrak{L}_1+\mathfrak{L}_2 \big)\rho^\varepsilon,
\end{align}
where $\mathfrak{L}_1$ denotes the differential operator containing all partial derivatives with respect to $x$ and similar $\mathfrak{L}_2$ is the differential operator containing all partial derivatives with respect to $y$.
To close the system we introduce the following boundary and initial condition
\begin{align}\label{eq: FPE bc}
  \rho^\varepsilon(t,x,-R)&=0=\rho^\varepsilon(t,x,R)\quad  (t,x)\in (0,T]\times \mathbb{R},\\ \label{eq: FPE ic}
  \rho^\varepsilon(0,x,y)&=\rho^\varepsilon_0(x,y),\qquad (x,y)\in \mathbb{R}\times(-R,R),
\end{align}
where $\rho^\varepsilon_0$ denotes the initial distribution of the probability density.

\begin{remark}
    In this work we focus on the case that we are in a two-dimensional domain with absorbing boundary conditions, i.e. zero Dirichlet boundary conditions, in the slow direction.
    The study of the problem on a higher dimensional domain as well as other boundary conditions such as reflecting walls or periodic boundary conditions are also possible and require little change in the stochastic reduction method.
\end{remark}


\begin{remark}
We observe that the fast-slow behaviour from the stochastic differential equation is also present in the Fokker-Planck equation, where the fast and slow variables $x$ and $y$ represent now the spatial directions.
That is, the Fokker-Planck equation exhibits a strong heterogeneity and anisotropy along these directions characterized by the scaling of the two differential operators $\mathfrak{L}_1$ and $\mathfrak{L}_2$.
\end{remark}

Taking the formal limit as $\varepsilon\to 0$  we observe that the Fokker-Planck equation \eqref{eq:fast-slow FPE} degenerates and decouples into  a part on the critical manifold, which in this case is an invariant stationary distribution with respect to the fast direction/ variable $x$, and a slow evolution along the slow direction/ variable $y$
\begin{align}
    \partial_t \rho^0&=- \partial_y (g(x,y)\rho^0)+ \frac{(\sigma')^2}{2}\partial_{y,y}\big(G(x,y)^2\rho^0\big),
    \end{align}
where $\rho^0$ is defined on the critical manifold
\begin{align}
    S_0=\{\rho^0:~0&=  -\partial_x (f(x,y)\rho^0)+ \frac{\sigma^2}{2} \partial_{x,x} \big(F(x,y)^2\rho^0\big)\}.
\end{align}

For the subsequent sections we will make the several assumptions that we collect below. 
\begin{assumption}
\
\begin{enumerate}[ label={(A\arabic*)},ref=A\arabic*]
    \item \label{ass:1} The first one concerns the dynamics of the deterministic part of fast-slow SDE \eqref{eq: fast-slow SDE}. This allows us to connect the results of the existence of slow manifolds in this paper with the ones in the deterministic setting.
    Thus, we assume that the deterministic part of system \ref{eq: fast-slow SDE} has a slow manifold. 
    To be more precise we require the function $f$ to be normally hyperbolic with respect to the fast variable $x$, i.e., if the partial derivative of $f$ with respect to the fast variables has no zero $\textnormal{D}_x f(x,y)\neq 0$ for all $x,y$ satisfying $f(x,y)=0$.

    \item \label{ass:2} The second assumption specifies a property of the differential operator $\mathfrak{L}_1$ introduced in \eqref{eq:FPE short form}.
    We assume that the (singular) Sturm-Liouville problem
    \begin{align}\label{eq: sturm liouville}
        \mathfrak{L}_1 \varphi(x)= -\lambda \varphi(x)
    \end{align}
    is well-posed in $L^2(\mathbb{R})$, where $-\lambda\in \mathbb{C}$ is an eigenvalue of the operator $\mathfrak{L}_1$. 
    Moreover, we assume that the real part of the spectrum of the operator $-\mathfrak{L}_1$ is contained in $(-\infty,0]$ and that the corresponding eigenfunctions form an orthonormal basis $\{\phi_j\}_{j\in \mathbb{N}}$ of $L^2(\mathbb{R})$.

    \item \label{ass: solvability FPE}
    We assume that the diffusion terms $F,G$ in \eqref{eq: fast-slow SDE} are constant, i.e., we only have additive noise in the stochastic system. In addition we assume that the drift functions $f,g$ satisfy 
    $f,g\in C^2(\mathbb{R}\times [-R,R])$ such that $f,g,\nabla f,\nabla g$ are bounded in $\mathbb{R}\times [-R,R]$.
\end{enumerate}
\end{assumption}


\begin{remark}
    The Assumption \ref{ass:2} can also be phrased in terms of an ergodicity condition on the stochastic differential equation \eqref{eq: fast-slow SDE}, which holds when the diffusion matrix $F(x,y)$ is strictly positive-definite. See \cite{mattingly2002ergodicity,pavliotis2008multiscale} for further details.
    However, the formulation via the Sturm-Liouville problem allows for an explicit computation of the eigenbasis, which we perform for a linear example in Section \ref{sec: examples}.
\end{remark}


Then, the existence of solutions to the Fokker-Planck equation \eqref{eq:fast-slow FPE}, depending on the domain $\Omega$ and the regularity of the initial data $\rho_0$, is a well-studied result see \cite{cerrai2001second,bogachev2022fokker} and the references therein.\\

Now, the aim is to derive a reduced Fokker–Planck equation while maintaining the geometric structure of the original equation.
The major steps in deriving this reduced Fokker–Planck equation:
\begin{itemize}
    \item[(S1)] Define a projection operator onto the $y$-variables using the stationary distribution of $x$.
    \item[(S2)] Apply the projections and  use an eigenfunction decomposition of the probability distribution $\rho$ to eliminate the $x$-variable.
    \item[(S3)] Derive an infinite dimensional fast-slow system for the coefficient functions, coming from the eigenfunction decomposition, which depend only on the slow direction $y$ and time $t$.
    \item[(S4)] Apply the slow manifold theory for fast-slow PDEs in this setting and truncate the dimension of the fast-slow system to derive the reduced dynamics.
\end{itemize}

\subsection{Step (S1): Projection Operators}\label{sec: Step 1}

The first step in the stochastic reduction method is to split the Fokker-Planck equation \eqref{eq:fast-slow FPE} into a fast and a slow component.
To this end we introduce the projection operator $\PP:L^2(\mathbb{R})\to L^2(\mathbb{R})$ defined by
\begin{align}
    \PP \rho(t,x,y)= p_s(x) \int_{\mathbb{R}} \rho(t,\tilde x,y)\,\textnormal{d} x,
\end{align}
where $p_s(x)$ is the so-called stationary density of the system and solves
\begin{align}\label{eq: stationary distr}
    \mathfrak{L}_1 p_s(x)=0,\quad \int_{\mathbb{R}} p_s(x)\,\textnormal{d} x=1.
\end{align}
In the framework of geometric singular perturbation theory we call the stationary distribution the critical manifold of the system.
Thus, the operator $\PP$ projects onto the null space of the differential operator $\mathfrak{L}_1$ and thereby eliminates the fast term in the Fokker-Planck equation.
Similar, we introduce its complement $\QQ:L^2(\mathbb{R})\to L^2(\mathbb{R})$ as $\QQ=\textnormal{I}-\PP$.

\begin{remark}\label{rmk: stationary distr}
Note, that we can explicitly compute the stationary distribution $p_s(x)$.
Taking into account the boundary condition of the Fokker-Planck equation at infinity this reduces to solving a first order ODE in the $x$-component, where $y$ is seen as a parameter
\begin{align*}
    \frac{\textnormal{d}}{\textnormal{d}x}p_s(x)= -\frac{2f(x,y)}{\sigma_1^2}  p_s(x) = \psi(x,y) p_s(x).
\end{align*}
Denoting by $\Psi(x,y)$ the primitive of $\psi(x,y)$ with respect to the variable $x$ we obtain
\begin{align}\label{eq: stationary distr 2}
    p_s(x)= c\, \textnormal{e}^{\Psi(x,y)},
\end{align}
where $c$ is chosen according to the mass constraint \eqref{eq: stationary distr}.
It follows that $p_s(x)\geq 0$ for all $(x,y)\in \Omega$.
\end{remark}

Before applying the projections $\PP$ and $\QQ$ to the Fokker-Planck equation we make the following observations on the properties of these projection operators.
\begin{align*}
    \PP^2 \rho(x,y)= p_s(x) \int_{\mathbb{R}} \int_{\mathbb{R}} p_s(\tilde x)\rho(\bar x,y) \,\textnormal{d}\bar x \,  \textnormal{d}\tilde x= p_s(x)\int_{\mathbb{R}} \rho(\bar x,y)\,\textnormal{d}\bar x = \PP \rho(x,y) ,
\end{align*}
i.e., $\PP^2=\PP$.
From this we can further deduce that $\QQ^2=\QQ$ and that $\QQ\PP=\PP\QQ=0$.
In addition,
\begin{align*}
    \PP \mathfrak{L}_1 \rho(x,y)&= p_s(x) \int_{\mathbb{R}}  \mathfrak{L}_1 \rho(\tilde x,y)\, \textnormal{d}\tilde x = p_s(x)\int_{\mathbb{R}} \partial_{\tilde x}\big( \partial_{\tilde x}\frac{1}{2}\sigma^2-f(\tilde x,y)\big)\rho(\tilde x,y)\,\textnormal{d}\tilde x =0,
\end{align*}
where we used the natural boundary conditions for $\rho$ in $x$-direction.
Similarly, we have
\begin{align*}
    \mathfrak{L}_1 \PP \rho(x,y)= \mathfrak{L}_1 p_s(x)\int_{\mathbb{R}}  \rho(\tilde x,y)\, \textnormal{d}\tilde x=0
\end{align*}
by the definition of $p_s$.
Hence, $\PP \mathfrak{L}_1=\mathfrak{L}_1\PP=0$.
For the action of $\PP$ on the operator $\mathfrak{L}_2$ we observe that
\begin{align*}
  \PP  \mathfrak{L}_2 \PP \rho(x,y)&= p_s(x) \int_{\mathbb{R}} \int_{\mathbb{R}}  \mathfrak{L}_2\big( p_s(\tilde x)  \rho(\bar x,y)\,\textnormal{d} \bar x\big) \,  \textnormal{d}\tilde x\\
  &= p_s(x) \int_{\mathbb{R}} \int_{\mathbb{R}} \big(\frac{\sigma_2^2}{2}\partial_{y,y}-\partial_y g(\tilde x,y)\big) \big(p_s(\tilde x) \rho(\bar x,y)\big) \,\textnormal{d} \bar x \,\textnormal{d}\tilde x \\
  &=p_s(x) \int_{\mathbb{R}} \int_{\mathbb{R}}\rho(\bar x,y) \big(\frac{\sigma_2^2}{2}\partial_{y,y}-\partial_y g(\tilde x,y)\big) \big(p_s(\tilde x)\big)  \,  \textnormal{d} \bar x \, \textnormal{d}\tilde x\\
  &\quad +p_s(x)  \int_{\mathbb{R}}  \frac{\sigma_2^2}{2}\partial_{y,y}\rho(\bar x,y)\, \textnormal{d} \bar x - p_s(x) \int_{\mathbb{R}} \int_{\mathbb{R}} p_s(\tilde x)\rho(\bar x,y)\partial_y \big(g(\tilde x,y)\big)   \, \textnormal{d} \bar x \,\textnormal{d}\tilde x \\
  &\quad - p_s(x) \int_{\mathbb{R}} \int_{\mathbb{R}} p_s(\tilde x) g(\tilde x,y) \partial_y \rho(\bar x,y) \, \textnormal{d} \bar x \,\textnormal{d}\tilde x ,
\end{align*}
where the first term only equals to zero when $f(x,y)=f(x)$, i.e. the solution of $\mathfrak{L}_1 p_s(x)$ does not depend on the parameter $y$.
We note that the last three terms represent an averaging over the $x$-direction while applying the differential operator $\mathfrak{L}_2$.

The last observation we make is that both $\PP$ and $\QQ$ commute with the partial time derivative as the projections do not depend on the time.

\subsection{Step (S2): Eigenfunction Decomposition} \label{sec: Step 2}

Setting $\PP\rho=v$ and $\QQ\rho=w$ we can apply the projection operators $\PP$ and $\QQ$ to the Fokker-Planck equation and obtain
\begin{align}\label{eq: v equation}
    \partial_t v^\varepsilon&= \PP \mathfrak{L}_2 (v^\varepsilon+w^\varepsilon),\\ \label{eq: w equation}
    \partial_t w^\varepsilon&= \frac{1}{\varepsilon}\mathfrak{L}_1 w^\varepsilon+ \QQ \mathfrak{L}_2(v^\varepsilon+w^\varepsilon).
\end{align}

In the language of fast-slow systems we call $w^\varepsilon$ the fast variable and $v^\varepsilon$ the slow variable.
Taking the formal limit $\varepsilon\to 0$ in the above system yields
\begin{align}
    \partial_t v^0&= P\mathfrak{L}_2 v^0,\\
    w^0 &=0.
\end{align}
The last equation follows from the fact that in the limit $w^0$ has to satisfy $\mathfrak{L}_1 w^0=0$. 
This can only be the case if $w^0=0$, since $w^0=\QQ \rho^0$ and $\QQ$ is orthogonal to the null space of $\mathfrak{L}_1$.
Hence, $w^\varepsilon$ can be understood as an additive perturbation to the invariant density or critical manifold $p_s$.\\

Moreover, we observe that we can rewrite the $v^0$-equation such that only the $y$-component remains.
That is, 
\begin{align*}
    0=p_s(x)\int_{\mathbb{R}}  \big(\partial_t-\mathfrak{L}_2 \big)v^0(t,\tilde x,y)\, \textnormal{d}\tilde x = p_s(x)\big(\partial_t-\bar {\mathfrak{L}}_2\big) \bar v^0(t,y),
\end{align*}
where $\bar{\mathfrak{L}}_2$ denotes the averaged (with respect to the fast direction) differential operator $\mathfrak{L}_2$ and where $\bar v^0(t,y)=\int_{\mathbb{R}} v^0(t,\tilde x,y)\, \textnormal{d}\tilde x$ is the marginal density.
Hence, we have reduced the full system to an equation for the slow variable
\begin{align}
    \big(\partial_t-\bar {\mathfrak{L}}_2\big) \bar v^0(t,y)=0.
\end{align}


Next, we further investigate the properties of the operator $\mathfrak{L}_1$, where under the assumptions of \ref{sec: problem set up}, in particular Assumption \ref{ass:2}, we have that there exists an eigenfunction decomposition in $x$-direction of the probability density $\rho(t,x,y)$, which we write as 
\begin{align}\label{eq: eigenfunction sum}
    \rho(t,x,y)= \sum_{j\in \mathbb{N}} a_j(t,y) \phi_j^y(x),
\end{align}
where $\phi_j^y(x)$ is the eigenfunction corresponding to the eigenvalue $\lambda_j$ of the Sturm-Liouville equation \eqref{eq: sturm liouville} and $j\in \mathbb{N}=0,1,2,\dots$.
Here, the superscript denotes the parametric dependence on the variable $y$ of the eigenfunctions $\phi_j^y(x)$.
Note, that we can express the stationary density $p_s(x)$ in terms of the eigenfunction $\phi_0^y(x)$ by $p_s(x)=\phi_0^y(x)^2$.

Applying the projections $\PP$ and $\QQ$ from step (S1) to the eigenfunction decomposition \eqref{eq: eigenfunction sum} we obtain that
\begin{align}\label{eq: fast slow decomposition}\begin{split}
    v(t,x,y)&=a_{0}(t,y) \phi^y_0(x) , \\
    w(t,x,y)&=\sum_{j\in \mathbb{N}\setminus \{0\}} a_{j}(t,y) \phi^y_j(x) ,\end{split}
\end{align}
where we study the coefficient functions $(a_j(t,y))_{j\in \mathbb{N}}$ in the next step.

\subsection{Step (S3): Coefficient System} \label{sec: Step 3}

The idea in this section is to obtain a system of PDEs for the coefficient functions $a_0=a_0(t,y)$ and $a_j=a_j(t,y)$ for $j=1,2,3,\dots$, derived from the fast-slow Fokker-Planck equation \eqref{eq:fast-slow FPE}.

\begin{lemma}
    The coefficient functions derived from the splitting of the fast-slow Fokker-Planck equation \eqref{eq:fast-slow FPE} satisfy the following system of equations
    \begin{align}\label{eq:coefficient slow}
        \begin{split}
            \partial_t   a_0(t,y)&=C_0\frac{\sigma_2^2}{2}\partial_{y,y} a_0(t,y) -  \partial_y (G_{0}(y)  a_0(t,y)) -  \sum_{i \in \mathbb{N}\setminus \{0\}} \partial_y(G_{i}(y) a_{i}(t,y)),
        \end{split} 
        \end{align}
        \begin{align}\label{eq:coefficient fast}
        \begin{split}
              \partial_t a_j(t,y)  &= \frac{\sigma_2^2}{2}  \partial_{y,y} a_j(t,y) - \frac{\lambda_j}{\varepsilon}a_j(t,y)  - \partial_y \big( G_{0,j}(y) a_0(t,y)\big) - \sum_{i \in \mathbb{N}\setminus \{0\}}\partial_y \big( G_{i,j}(y) a_i(t,y)\big)\\
        &\quad + \tilde G_{0,j}(y)a_0(t,y)+ \sum_{i \in \mathbb{N}\setminus \{0\}}\tilde G_{i,j}(y) a_i(t,y),
        \end{split}
    \end{align}
where the functions $G(y),\tilde G(y)$ will be defined in the proof.
\end{lemma} 

\begin{proof}
We start with the decomposition of the probability density $\rho^\varepsilon$ in equation \eqref{eq: fast slow decomposition} and plug it into the system \eqref{eq: v equation}-\eqref{eq: w equation}.
Then, integrating the slow equation \eqref{eq: v equation} with with respect to the fast direction $x$ yields 
\begin{align*}
  \partial_t   a_0(t,y)=  \int_{\mathbb{R}} \mathfrak{L}_2 \bigg[a_0(t,y)\phi^y_0(x)   +\sum_{i \in \mathbb{N}\setminus \{0\}} a_{i}(t,y) \phi^y_i(x)  \bigg]\,\textnormal{d} x,
\end{align*}
and for the fast equation \eqref{eq: w equation} multiplying it with $\phi_j(x)$ and integrating with respect to $x$, yields
\begin{align*}
\begin{split}
    \partial_t a_j(t,y) &=- \frac{\lambda_j}{\varepsilon}a_j(t,y)+ \int_{\mathbb{R}} \mathfrak{L}_2\bigg( a_0(t,y) \phi^y_0(x)  +\sum_{i \in \mathbb{N}\setminus \{0\}} a_{i}(t,y) \phi^y_i(x)  \bigg)\phi_j^y(x)\,\textnormal{d} x\\
    &\quad -\int_{\mathbb{R}} \int_{\mathbb{R}}  \phi_0^y(x)^2\phi_j^y(x)  \mathfrak{L}_2 \bigg(a_0(t,y)\phi^y_0(\tilde x)   +\sum_{i \in \mathbb{N}\setminus \{0\}} a_{i}(t,y) \phi^y_i(\tilde x)  \bigg)\,\textnormal{d}\tilde x\,\textnormal{d} x,
    \end{split}
\end{align*}
where $\lambda_j$ is the eigenvalue corresponding to the eigenfunction $\phi_j^y(x)$.

In the next step we need to analyze this system further.
We note, that we can interchange the order of differentiation with respect to $y$ and integration with respect to $x$ in the above equations.
Moreover, we define for all $k,j \in \mathbb{N}$
\begin{align}\label{eq:definition G}
\begin{split}
      C_0&=\int_{\mathbb{R}}\phi_0^y(x)\,\textnormal{d} x, \quad G_{k}(y)=\int_{\mathbb{R}} g(x,y)  \phi_k^y(x) \,\textnormal{d} x,\\
      G_{k,j}(y)&= \int_{\mathbb{R}} g(x,y)  \phi_k^y(x)\phi_j(x)\,\textnormal{d} x  ~\textnormal{ and }~ \tilde G_{k,j} (y)=  \int_{\mathbb{R}} g(x,y)  \phi_k^y(x) \partial_y \phi_j(x)\,\textnormal{d} x.
\end{split}
\end{align}
Using that $p_s(x)= \phi_0^y(x)^2$ it follows from remark \ref{rmk: stationary distr} that $C_0>0$.
Moreover, we observe that 
\begin{align*}
    \int_{\mathbb{R}}   \phi_0^y(x)^2\phi_j^y(x)\, \textnormal{d}x= 0,\quad \textnormal{and}\quad \int_{\mathbb{R}}\phi_j^y(x)\,\textnormal{d} x=0 
\end{align*}
for all $j\in \mathbb{N}\setminus\{0\}$.
This follows from the properties of the spectral eigenfunctions $\phi_j^y(x)$, in particular from the fact that $\{\phi_j\}_{j\in \mathbb{N}}$ form an orthonormal basis in $L^2(\mathbb{R})$.
Then, we obtain 
\begin{align*}
   \partial_t   a_0(t,y)&=  \int_{\mathbb{R}} \big(\frac{\sigma_2^2}{2}\partial_{y,y}-g(x,y)\partial_y- \partial_y g(x,y)\big) \bigg(a_0(t,y)\phi^y_0(x)   +\sum_{i \in \mathbb{N}\setminus \{0\}} a_{i}(t,y) \phi^y_i(x)  \bigg)\,\textnormal{d} x \\
   &= \frac{\sigma_2^2}{2}\partial_{y,y}\bigg(C_0 a_0(t,y) + \sum_{i \in \mathbb{N}\setminus \{0\}} a_{i}(t,y) \int_{\mathbb{R}}\phi^y_i(x)  \,\textnormal{d} x\bigg)\\
   &\quad - \partial_y   a_0(t,y) \int_{\mathbb{R}} g(x,y)\phi^y_0(x) \,\textnormal{d} x -  \sum_{i \in \mathbb{N}\setminus \{0\}}\partial_y a_{i}(t,y)\int_{\mathbb{R}} g(x,y)  \phi^y_i(x) \,\textnormal{d} x\\
   &\quad - a_0(t,y)\int_{\mathbb{R}} g(x,y) \partial_y (\phi^y_0(x)  )\,\textnormal{d} x-\sum_{i \in \mathbb{N}\setminus \{0\}} a_{i}(t,y)\int_{\mathbb{R}} g(x,y)  \partial_y (\phi^y_i(x) )\,\textnormal{d} x\\
    &\quad - a_0(t,y)\int_{\mathbb{R}}\partial_y g(x,y)\phi^y_0(x)\,\textnormal{d} x  - \sum_{i \in \mathbb{N}\setminus \{0\}} a_{i}(t,y)\int_{\mathbb{R}}  \partial_y g(x,y)  \phi^y_i(x) \,\textnormal{d} x.
\end{align*}
Further reduction of the above equation yields
\begin{align*}
\begin{split}
     \partial_t   a_0(t,y)&=C_0\frac{\sigma_2^2}{2}\partial_{y,y} a_0(t,y) -  \partial_y (G_{0}(y)  a_0(t,y)) -  \sum_{i \in \mathbb{N}\setminus \{0\}} \partial_y(G_{i}(y) a_{i}(t,y)).
     \end{split}
\end{align*}
Similar we obtain
\begin{align*}
    \partial_t a_j(t,y)  &=- \frac{\lambda_j}{\varepsilon}a_j(t,y) + \int_{\mathbb{R}} \mathfrak{L}_2\bigg( a_0(t,y) \phi^y_0(x)  +\sum_{i \in \mathbb{N}\setminus \{0\}} a_{i}(t,y) \phi^y_i(x)  \bigg) \phi_j(x)\,\textnormal{d} x.
\end{align*}
This equation can be simplified as
\begin{align*}
    \begin{split}
        \partial_t a_j(t,y)  &= \frac{\sigma_2^2}{2}  \partial_{y,y} a_j(t,y) - \frac{\lambda_j}{\varepsilon}a_j(t,y)  - \partial_y \big( G_{0,j}(y) a_0(t,y)\big) - \sum_{i \in \mathbb{N}\setminus \{0\}}\partial_y \big( G_{i,j}(y) a_i(t,y)\big)\\
        &\quad + \tilde G_{0,j}(y)a_0(t,y)+ \sum_{i \in \mathbb{N}\setminus \{0\}}\tilde G_{i,j}(y) a_i(t,y).
    \end{split}
\end{align*}
\end{proof}
To close the system, we recall the boundary and initial conditions from the original system.
In this setting the absorbing or zero Dirichlet boundary conditions are given by
\begin{align}\label{eq: coefficient bc absorbing}
     a_j(t,-R)=0=a_j(t,R),\quad t\in (0,T] \textnormal{ and for all } j\in \mathbb{N},
\end{align}
and the initial conditions are given by
\begin{align}
    \label{eq: coefficient ic}
    a_j(0,y)=a_{j,0}(y), \quad y\in (-R,R) \textnormal{ and for all } j\in \mathbb{N},
\end{align}
where the initial data relates back to the initial data from the original Fokker-Planck equation via
\begin{align*}
    \sum_{j\in \mathbb{N}}a_{j,0}(y)\phi_j^y(x)=\rho_0(x,y).
\end{align*}
Hence we have transformed the Fokker-Planck equation into a fast-slow system with linear leading order.

Again, we can observe that in the formal limit as $\varepsilon\to 0$ we obtain $a_j(t,y)=0$ for all $j\in \mathbb{N}\setminus \{0\}$, which corresponds to $w(t,x,y)=0$.

\begin{remark}
Note that the system of equations \eqref{eq:definition G}-\eqref{eq: coefficient ic} is a system of infinitely many coupled PDEs for the coefficient functions $a_j$ with $j\in \mathbb{N}$.   
\end{remark}
\begin{remark}
    In Section \ref{sec: examples} we will consider examples, where we can further reduce the above system . In particular, equation \eqref{eq:definition G} can be simplified, for example, when $g(x,y)=g(y)$ or when we have more information on the eigenfunctions $\{\phi_j\}_{j\in \mathbb{N}}$.
\end{remark}


\begin{lemma}\label{lemma: general existence coef}
    Under the assumptions of Section \ref{sec: problem set up} the system \eqref{eq:coefficient slow}-\eqref{eq:coefficient fast} together with equations \eqref{eq:definition G}-\eqref{eq: coefficient ic} is well-posed and the solutions satisfy 
    $$a_j(t,y)\in C^{1}((0,T);L^2(-R,R))\cap C([0,T);H^2(-R,R))\quad \text{for all~~} j\in \mathbb{N} \text{ and any fixed }~T>0.$$
\end{lemma}

\begin{proof}
As the equations are obtained via a series of equivalent transformations from the original Fokker-Planck equation \eqref{eq:fast-slow FPE} it suffices to prove the existence of solutions for the original equation.
This however follows immediately from assumption \ref{ass: solvability FPE} as we now can apply classical results for parabolic equations on an infinite strip \cite{watson1989parabolic}.
And for more details on the existence of solutions for the Fokker-Planck we refer to \cite{bogachev2022fokker} and the references therein.
\end{proof}

Thus, the system of equations \eqref{eq:coefficient slow}-\eqref{eq: coefficient ic} is well-behaved.
Moreover, we can study its dynamics and geometry in order to reduce the dimension of the system by using the abstract Banach space theory presented in the next section of this paper.

\begin{remark}
    We want to emphasize that up to this point the stochastic reduction method would also work in higher dimensional spatial domains as well as with different domains and boundary conditions.
\end{remark}

\begin{remark}\label{rmk: different regularity of solutions}
In applications it is not always possible to satisfy the assumptions on the regularity of the functions $f,g$, which is needed for the geometric study of the fast-slow system in the next section.
Therefore, we want to recall some of the results for the existence of a solution to the Fokker-Planck equation for less regular coefficients $f,g$.
    \begin{itemize}
        \item The first result is due to \cite{jordan1998variational}, where the authors assume that there exists a potential $\Psi\in C^\infty(\mathbb{R}^2)$ such that $(f,g)^T=\nabla \Psi$ satisfying $\Psi(x,y)\geq 0$ and $|\nabla\Psi(x,y)|\leq C(\Psi(x,y)+1)$ for all $(x,y)\in \mathbb{R}^2$.
        Then for any initial data $\rho_0\in L^1(\mathbb{R}^2)$ there exists a probability density $\rho\in L^1((0,T)\times \mathbb{R}^2)$ solving the fast-slow Fokker-Planck equation \eqref{eq:fast-slow FPE} such that $\int_{\mathbb{R}^2} |(x,y)|^2\rho(t,x,y)|\textnormal{d}x\textnormal{d}y <\infty$ and $\int_{\mathbb{R}^2}\Psi(x,y)\rho(t,x,y)|\textnormal{d}x\textnormal{d}y <\infty$ for all $t\in (0,T)$.
        \item The next result \cite{le2004renormalized,bris2008existence} assumes that $f,g\in W^{1,1}_{\textnormal{loc}}(\mathbb{R}^2)$, $\textnormal{div}f,\textnormal{div}g\in L^\infty( \mathbb{R}^2)$ and $\frac{f,g}{1+|(x,y)|}\in L^\infty(\mathbb{R}^2)$. 
        Then for all $\rho_0\in L^2(\mathbb{R}^2)\cap L^\infty(\mathbb{R}^2)$ there exists a solution $\rho$ to the fast-slow Fokker-Planck equation satisfying $\rho\in L^\infty([0,T];L^2(\mathbb{R}^2)\cap L^\infty(\mathbb{R}^2))$ and $\nabla \rho \in L^2([0,T];L^2(\mathbb{R}^2))$.
        \item The last result we consider is due to \cite{bogachev2008parabolic,manita2016cauchy,bogachev2022fokker}, where the authors assume that $f,g\in L^\infty([0,T]\times B)$ for any ball $B\subset \mathbb{R}^2$. Then there exists a solution to the Fokker-Planck equation $\rho\in L^{3/2}_{\textnormal{loc}}([0,T]\times\mathbb{R}^2)$.
    \end{itemize}
\end{remark}



\section{Fast Diffusion Systems}\label{sec:abstract theory}

In this section we prove the existence of a slow manifold for a fast diffusion system that has a similar structure to the Fokker-Planck system of the previous section after using the stochastic reduction method. 
The general form of the equations we consider is
\begin{align}\label{eq: fast diffusion}
    \begin{split}
        \partial_t u^\varepsilon &= \frac{1}{\varepsilon} A u^\varepsilon + \mathcal{F}(u^\varepsilon, v^\varepsilon),\\
        \partial_t v^\varepsilon&= B v^\varepsilon + \mathcal{G}(u^\varepsilon, v^\varepsilon),
    \end{split}
\end{align}
where $0< \varepsilon\ll 1$ is a small parameter that encompasses the separation of time scales, and where $A,B$ are linear and possibly unbounded operators on a Banach space $X$, respectively $Y$.
The initial data satisfies
\begin{align}
    u^\varepsilon(0)&=  u^\varepsilon_0 \in X_1,\qquad  v^\varepsilon(0)=  v^\varepsilon_0\in Y_1,
\end{align}
where  $X_1= D(A)$ denotes the domain of the operator $A$ and $Y_1= D(B)$ the domain of $B$.
The formal limit of the system as $\varepsilon\to 0$ is
\begin{align}\label{eq: fast diffuion limit}
   \begin{split}
       0&= A u^0,\\
       \partial_t v^0&= B v^0+ \mathcal{G}(u^0,v^0),\\
       v^0(0)&=v_0,
   \end{split}
\end{align}
provided the real part of the spectrum of the operator $A$ is entirely negative and where the critical manifold $S_0$ is given as the null space of the operator $A$, i.e.
\begin{align}\label{eq: critical manifold}
    S_0:=\{ u\in X: ~Au =0\}.
\end{align}

\subsection{Assumptions}\label{Sec:assumptions}
For the operators $A$ and $B$ we make the following assumptions
\begin{itemize}
    \item $A, B$ are closed linear operators, where $A:X\supset D(A)\to X$ generates the exponentially stable $C_0$-semigroup $(\textnormal{e}^{tA})_{t\geq 0}\subset \mathcal{B}(X)$ and $B:Y\supset D(B)\to Y$ generates the $C_0$-semigroup $(\textnormal{e}^{tB})_{t\geq 0}\subset \mathcal{B}(Y)$, respectively, where $\mathcal{B}(\cdot)$ denotes the space of bounded linear operators.
    \item The interpolation-extrapolation scales generated by $(X,A)$ and $(Y,B)$ are given by $(X_\alpha)_{\alpha\in [-1,\infty)}$ and $(Y_\alpha)_{\alpha\in [-1,\infty)}$ (for details regarding the definitions and basic properties of interpolation-extrapolation scales see \cite{amann1995linear}).
    \item Let $\gamma\in (0,1]$ if $(\textnormal{e}^{tA})_{t\geq 0}$ is analytic and set $\gamma=1$ otherwise.
    Similarly, we introduce $\delta \in (0,1]$ if $(\textnormal{e}^{tB})_{t\geq 0}$ is analytic and set $\delta=1$ otherwise.
     \item There are constants $C_A,C_B, M_A, M_B>0$, $\omega_A <0 $ and $\omega_B\in \mathbb{R}$ such that
    \begin{align}\label{eq: semigroup estimate}
    \begin{split}
        \|\textnormal{e}^{tA}\|_{\mathcal{B}(X_{1})}&\leq M_A \textnormal{e}^{\omega_A t},\qquad \|\textnormal{e}^{tA}\|_{\mathcal{B}(X_\gamma,X_{1})}\leq C_At^{\gamma-1}\textnormal{e}^{\omega_A t},\\
        \|\textnormal{e}^{tB}\|_{\mathcal{B}(Y_{1})}&\leq M_B \textnormal{e}^{\omega_B t},\qquad \|\textnormal{e}^{tB}\|_{\mathcal{B}(Y_\delta,Y_{1})}\leq C_B t^{\delta-1} \textnormal{e}^{\omega_B t}
        \end{split}
    \end{align}
    hold for all $t\in (0,\infty)$.
\end{itemize}
For the nonlinear functions $\mathcal{F}$ and $\mathcal{G}$ we have the following assumptions
\begin{itemize}
    \item The nonlinearities $\mathcal{F}:X_1\times Y_{1}\to X_{\gamma}$ and $\mathcal{G}:X_{1}\times Y_{1}\to Y_\delta$ are Fr\'echet differentiable. Moreover, there are constants $L_\mathcal{F}, L_\mathcal{G}>0$ such that
    \begin{align}
       \|\textnormal{D} \mathcal{F}(x,y)\|_{\mathcal{B}(X_{1}\times Y_{1},X_\gamma)}&\leq L_\mathcal{F},\qquad  \|\textnormal{D} \mathcal{G}(x,y)\|_{\mathcal{B}(X_{1}\times Y_{1},Y_\delta)}\leq L_\mathcal{G}, 
    \end{align}
    for all $(x,y)\in X_1\times Y_1$.
    This implies that the following inequalities hold
        \begin{align*}
        \|\mathcal{F}(x_1,y_1)-\mathcal{F}(x_2,y_2)\|_{X_\gamma}&\leq L_\mathcal{F}( \|x_1-x_2\|_{X_{1}}+\|y_1-y_2\|_{Y_{1}}),\\
        \|\mathcal{G}(x_1,y_1)-\mathcal{G}(x_2,y_2)\|_{Y_\delta}&\leq L_\mathcal{G}( \|x_1-x_2\|_{X_{1}}+\|y_1-y_2\|_{Y_{1}})
    \end{align*}
    for all $x_1,x_2\in X_{1}$, $y_1,y_2\in Y_{1}$.  
    \item For convenience, we assume that $\mathcal{F}(0,0)=0$ and $\mathcal{G}(0,0)=0$.
\end{itemize}

These assumptions guarantee the existence of strong solutions $(u^\varepsilon,v^\varepsilon)$ and $(u^0,v^0)$ in the space $C^1((0,T_*);X\times Y)\cap C([0,T_*);X_1\times Y_1)$, where $T_*$ denotes the maximal time of existence of the solutions. 
Moreover, we can write the solutions in the mild solution form using Duhamel's formula.
For more details we refer to \cite{lunardi2012analytic}.

Next, we require some technical assumptions on the Banach space $Y$ such that a splitting into fast and slow components is possible. The splitting is necessary since the operator $B$ is unbounded \cite{hummel2022slow}. Let $\zeta>0$ be a small parameter. Then, we introduce a splitting of the space $Y$ to split off effectively fast parts 
\begin{align*}
    Y= Y^\zeta_F\oplus Y^\zeta_S,
\end{align*}
where the splitting satisfies
\begin{itemize}
    \item The spaces $Y^\zeta_F$ and $Y^\zeta_S$ are closed in $Y$ and the projections $\pr_{Y^\zeta_F}$ and $\pr_{Y^\zeta_S}$ commute with $B$ on $Y_1$.
    \item The spaces $Y^\zeta_F\cap Y_1$ and $Y^\zeta_S\cap Y_1$ are closed subspaces of $Y_1$ and are endowed with the norm $\|\cdot\|_{Y_1}$.
    \item The realization of $B$ in $Y^\zeta_F$, i.e. 
    \begin{align*}
        B_{Y^\zeta_F}:D(Y^\zeta_F)&\subset Y^\zeta_F\to Y^\zeta_F,\quad v\mapsto Bv
        \intertext{with}
        D(Y^\zeta_F)&:= \{v_0\in Y^\zeta_F\,:\, Bv_0\in Y^\zeta_F\}
    \end{align*}
    has $0$ in its resolvent set.
    \item The realization of $B$ in $Y^\zeta_S$ generates a $C_0$-group $(\textnormal{e}^{tB_{Y^\zeta_S}})_{t\in \mathbb{R}}\subset \mathcal{B}((Y^\zeta_S,\|\cdot\|_{Y}) )$ which for $t\geq 0$ satisfies $\textnormal{e}^{tB_{Y^\zeta_S}}=\textnormal{e}^{tB}$ on $Y^\zeta_S$.
    \item The fast subspace $Y^\zeta_F$ contains the parts of $Y_1$ that decay under the semigroup $(\textnormal{e}^{tB})_{t\geq 0}$ almost as fast as functions in the fast variable space $X_1$ under the semi-group $(\textnormal{e}^{\zeta^{-1}t A})_{t\geq 0}$ generated by the operator $A$.
    The space $Y^\zeta_S$ on the other hand contains the parts of $Y_1$ that do not decay or which decay only slowly under the semigroup $(\textnormal{e}^{tB})_{t\geq 0}$ compared to $X_1$ under $(\textnormal{e}^{\zeta^{-1}t A})_{t\geq 0}$. 
    Hence, there are constants $C_B,\,M_B>0$ such that for $\zeta>0$ small enough there are constants $N_F^\zeta,N_S^\zeta$ satisfying   
    $$0\leq N_F^\zeta< N_S^\zeta\leq |\zeta^{-1}\omega_A|$$
    such that for all $t\geq 0$ and $y_F\in Y^\zeta_F,\, y_S\in Y^\zeta_S$ we have 
    \begin{align*}
        \|\textnormal{e}^{tB}y_F\|_{Y_1}&\leq C_B t^{\delta -1} \textnormal{e}^{(N_F^\zeta +\zeta^{-1}\omega_A)t}\|y_F\|_{Y_\delta},\\
        \|\textnormal{e}^{-tB}y_S\|_{Y_1}&\leq M_B \textnormal{e}^{-(N_S^\zeta +\zeta^{-1}\omega_A)t}\|y_S\|_{Y_1}.
    \end{align*}
    \item A necessary condition on the size of the parameter $\zeta>0$ is the following relation
         $$(1-\varepsilon \zeta^{-1})\omega_A-\frac{\varepsilon}{2}(N_S^\zeta+N_F^\zeta)<0.$$
        This can be achieved by assuming $\varepsilon\zeta^{-1}\leq 1$.
    \item Lastly, the following spectral gap condition holds
    \begin{align}\label{spectral gap condition}
         L_{\textnormal{spec}}:=  \frac{\varepsilon 2^\gamma L_\mathcal{F} C_A \Gamma(\gamma)}{(2(\varepsilon\zeta^{-1} -1)\omega_A + \varepsilon (N_S^\zeta+N_F^\zeta) )^\gamma} + \frac{2^\delta L_\mathcal{G}C_B \Gamma(\delta)}{(N_S^\zeta-N_F^\zeta)^\delta}+\frac{2L_\mathcal{G} M_B \Gamma(\delta)}{N_S^\zeta-N_F^\zeta}< 1.
    \end{align}
\end{itemize}
\begin{remark}
    Note, that, in contrast to the previous results \cite{hummel2022slow,kuehn2023fast}, we can in particular choose $\varepsilon=\zeta$. This is due to the special structure of the fast diffusion problem~\eqref{eq: fast diffusion} that we consider in this work. 
\end{remark}

\subsection{Convergence of Solutions}

The first step in the analysis of the fast-diffusion system is to show the convergence of solutions of system \eqref{eq: fast diffusion} to solutions of the limit system \eqref{eq: fast diffuion limit} as $\varepsilon\to 0$.
\begin{lemma}\label{lemma: convergnce of sol}
    Let the assumptions of Section \ref{Sec:assumptions} hold and assume in addition that the initial data satisfies $v^\varepsilon_0=v_0$.
    Then for all $t\in [0,T_*)$ the following convergence result holds
    \begin{align*}
      \| u^\varepsilon(t)\|_{X_1}+  \|v^\varepsilon(t)-v^0(t)\|_{Y_1}&\leq    C  e^{\varepsilon^{-1}\omega_A t} \|u^\varepsilon_0\|_{X_1}+ C \varepsilon^\delta \big( \|u^\varepsilon_0\|_{X_1}+\|v_0\|_{Y_1}\big),
    \end{align*}
    where $C>0$ is a constant depending on the Lipschitz constants $L_\mathcal{F},L_\mathcal{G}$, the constants arising in the semigroup estimates \eqref{eq: semigroup estimate} and the time $t$ but is independent of $\varepsilon$.
\end{lemma}
\begin{proof}
   We note that the assumptions on the operator $A$, in particular the requirement that the semigroup bound $\omega_A<0$, imply that $u^0\equiv 0$.
   Thus, we obtain
   \begin{align*}
       \| u^\varepsilon(t)-u^0\|_{X_1}&=  \bigg\|e^{\varepsilon^{-1}At}u_0+ \int_0^t e^{\varepsilon^{-1}A(t-s)} \mathcal{F}(u^\varepsilon(s),v^\varepsilon(s))\, \textnormal{d}s\bigg\|_{X_1}\\
       &\leq  C_A  e^{\varepsilon^{-1}\omega_A t} \|u^\varepsilon_0\|_{X_1}+ C_A L_\mathcal{F} \int_0^t \frac{e^{\varepsilon^{-1}\omega_A(t-s)}}{(t-s)^{1-\gamma}}\big(\|u^\varepsilon(s)\|_{X_1}+ \|v^\varepsilon(s)-v^0(s)\|_{Y_1}+\|v^0(s)\|_{Y_1}\big)\, \textnormal{d}s,
   \end{align*}
  where we used the semigroup estimate \eqref{eq: semigroup estimate} in the last inequality.\\
    Using the estimate $\|v^0(t)\|_{Y_1} \leq C \|v_0\|_{Y_1}$ and applying a special version of Gronwall's inequality (see Lemma 2.8 in \cite{hummel2022slow}) yields
   \begin{align*}
        \| u^\varepsilon(t)\|_{X_1}&\leq  C  e^{\varepsilon^{-1}\omega_A t} \|u^\varepsilon_0\|_{X_1}+\varepsilon C \|v_0\|_{Y_1} +\varepsilon C\sup_{0\leq s\leq t} \big(\|v^\varepsilon(s)-v^0(s)\|_{Y_1}\big).
   \end{align*}
   Similarly we estimate
\begin{align*}
    \|v^\varepsilon(t)-v^0(t)\|_{Y_1}&= \bigg\|\int_0^te^{B(t-s)}\big( \mathcal{G}(u^\varepsilon(s),v^0(s))-g(0,v^0(s))\big)\, \textnormal{d}s\bigg\|_{Y_1}\\
    &\leq C_B L_\mathcal{G}\int_0^t \frac{e^{\omega_B(t-s)}}{(t-s)^{1-\delta}}\big(\|u^\varepsilon(s)\|_{X_1}+ \|v^\varepsilon(s)-v^0(s)\|_{Y_1}\big)\, \textnormal{d}s.
\end{align*}
Inserting the previous estimate yields
\begin{align*}
    \|v^\varepsilon(t)-v^0(t)\|_{Y_1}&\leq C_B L_\mathcal{G}\int_0^t \frac{e^{\omega_B(t-s)}}{(t-s)^{1-\delta}}\big(  C  e^{\varepsilon^{-1}\omega_A s} \|u^\varepsilon_0\|_{X_1}+\varepsilon C \|v_0\|_{Y_1} \big)\, \textnormal{d}s\\
    &\quad + C_B L_\mathcal{G} \int_0^t \frac{e^{\omega_B(t-s)}}{(t-s)^{1-\delta}}\big(\varepsilon C\sup_{0\leq r\leq s} \|v^\varepsilon(r)-v^0(r)\|_{Y_1}+ \|v^\varepsilon(s)-v^0(s)\|_{Y_1}\big)\, \textnormal{d}s\\
    &\leq C \varepsilon^\delta \big( \|u^\varepsilon_0\|_{X_1}+\|v_0\|_{Y_1}\big) + C  \int_0^t \frac{e^{\omega_B(t-s)}}{(t-s)^{1-\delta}} \sup_{0\leq r\leq s} \|v^\varepsilon(r)-v^0(r)\|_{Y_1}\, \textnormal{d}s.
\end{align*}
   Noting that the right hand side of the inequality is increasing with time we can take the supremum with respect to the time $t$ on both sides.
   Then, applying Gronwall's inequality yields
   \begin{align*}
       \sup_{0\leq s\leq t} \|v^\varepsilon(s)-v^0(s)\|_{Y_1}&\leq  C \varepsilon^\delta \big( \|u^\varepsilon_0\|_{X_1}+\|v_0\|_{Y_1}\big)
   \end{align*}
    and thus proves the assertion.    
\end{proof}  

\subsection{Fenichel-Type Theorem}
To generalize Fenichel's Theorem for the existence of a slow manifold in this fast diffusion setting we first introduce a system with a splitting of the formal slow variable, i.e., the $v$-component, which gives
\begin{align}\label{eq:fast diffusion system splitting}
    \begin{split}
        \partial_t u^\varepsilon &= \frac{1}{\varepsilon}A u^\varepsilon + \mathcal{F}(u^\varepsilon,v_F^\varepsilon, v_S^\varepsilon),\\
        \partial_t v_F^\varepsilon &= B v_F^\varepsilon +\pr_{Y_F^\zeta} \mathcal{G}(u^\varepsilon,v_F^\varepsilon, v_S^\varepsilon),\\
        \partial_t v_S^\varepsilon &= B v_S^\varepsilon +\pr_{Y_S^\zeta} \mathcal{G}(u^\varepsilon,v_F^\varepsilon, v_S^\varepsilon),\\
        u^\varepsilon(0)&= u^\varepsilon_0,\quad  v^\varepsilon_F(0)= \pr_{Y_F^\zeta} v_0,\quad v^\varepsilon_S(0)= \pr_{Y_S^\zeta} v_0.
    \end{split}
\end{align}

First, we show the existence of an invariant manifold using the Lyapunov-Perron operator method.

\begin{lemma}\label{lemma: lyapunov perron}
  Under the assumptions of Section \ref{Sec:assumptions} the Lyapunov-Perron operator corresponding to system \eqref{eq:fast diffusion system splitting} has a fixed point.
\end{lemma} 
\begin{proof}
The main difference to the previous results on slow manifolds lies in the fast component (cf. \cite{hummel2022slow, kuehn2023fast}).
Therefore, we only consider the Lyapunov-Perron operator for this component and denote it by $\mathcal{L}_u$
\begin{align*}
    \mathcal{L}_u:C_\eta \to C_\eta,\quad     u^\varepsilon(t)\mapsto \bigg[ t\mapsto \int_{-\infty}^t e^{\varepsilon^{-1}A (t-s)} \mathcal{F}(u^\varepsilon,v_f^\varepsilon,v_S^\varepsilon)\, \textnormal{d}s \bigg].
\end{align*}
We aim to show the existence of a fixed point in the space 
\begin{align*}
C_\eta := C\big((-\infty, 0], \textnormal{e}^{\eta t};X1 \times (Y_F^\zeta \cap Y_1) \times ( Y_S^\zeta \cap Y_1)\big),    
\end{align*}
such that 
\begin{align*}
    \|(u,v_F,v_S)\|_{C_\eta}:= \sup_{t\leq 0}\textnormal{e}^{-\eta t}\big(\|u\|_{X_1}+\|v_F\|_{Y_1}+\|v_S\|_{Y_1}\big)<\infty,
    \end{align*}
where $\eta= \zeta^{-1}\omega_A +\frac{N_S^\zeta+N_F^\zeta}{2}<0$.
Then we estimate
\begin{align*}
    \|\mathcal{L}_u(u^\varepsilon,v_F^\varepsilon,v_S^\varepsilon)-\mathcal{L}_u(\tilde u^\varepsilon,\tilde v_F^\varepsilon,\tilde v_S^\varepsilon)\|_{C_\eta}&=\sup_{t\leq 0} e^{-\eta t}\bigg\|\int_{-\infty}^t e^{\varepsilon^{-1}A (t-s)} (\mathcal{F}(u^\varepsilon,v_f^\varepsilon,v_S^\varepsilon)-\mathcal{F}(\tilde u^\varepsilon,\tilde v_F^\varepsilon,\tilde v_S^\varepsilon))\, \textnormal{d}s\bigg\|\\
    &\leq C_A L_\mathcal{F} \varepsilon\int_{-\infty}^t \frac{(e^{\varepsilon^{-1}\omega_A -\eta)(t-s)}}{\varepsilon^\gamma(t-s)^{1-\gamma}}\, \textnormal{d}s \|u^\varepsilon-\tilde u^\varepsilon,v_F^\varepsilon-\tilde v_F^\varepsilon,v_S^\varepsilon-\tilde v_S^\varepsilon\|_{C_\eta}\\
&\leq \frac{\varepsilon C_AL_\mathcal{F} \Gamma(\gamma)}{|\omega_A-\varepsilon\eta|^\gamma} \|u^\varepsilon-\tilde u^\varepsilon,v_F^\varepsilon-\tilde v_F^\varepsilon,v_S^\varepsilon-\tilde v_S^\varepsilon\|_{C_\eta}\\
&\leq  \frac{\varepsilon 2^\gamma C_AL_\mathcal{F} \Gamma(\gamma)}{|2\omega_A(1-\varepsilon\zeta^{-1})-\varepsilon(N_S^\zeta+N_F^\zeta)|^\gamma} \|u^\varepsilon-\tilde u^\varepsilon,v_F^\varepsilon-\tilde v_F^\varepsilon,v_S^\varepsilon-\tilde v_S^\varepsilon\|_{C_\eta}.
\end{align*}
For $0<\varepsilon\leq\zeta\ll 1$ the above expression can be controlled by
\begin{align*}
     \|\mathcal{L}_u(u^\varepsilon,v_F^\varepsilon,v_S^\varepsilon)-\mathcal{L}_u(\tilde u^\varepsilon,\tilde v_F^\varepsilon,\tilde v_S^\varepsilon)\|_{C_\eta} \leq \varepsilon C \|u^\varepsilon-\tilde u^\varepsilon,v_F^\varepsilon-\tilde v_F^\varepsilon,v_S^\varepsilon-\tilde v_S^\varepsilon\|_{C_\eta}.
\end{align*}
We denote the fixed point of the Lyapunov-Perron operator evaluated at $t=0$ by $(h^{\varepsilon,\zeta}_{X_1}(v_S),h^{\varepsilon,\zeta}_{Y^\zeta_F}(v_S),v_S)$.
Then, the in-flowing invariant manifold $S_{\varepsilon,\zeta}$ is given by
\begin{align}
    S_{\varepsilon,\zeta}:=\{ \big(h^{\varepsilon,\zeta}_{X_1}(v_S),h^{\varepsilon,\zeta}_{Y^\zeta_F}(v_S),v_S\big):v_S\in Y_1\cap Y_S^\zeta\}.
\end{align}
\end{proof}
\begin{theorem}\label{thm: Fenichel abstract}
    Let the assumptions of Section \ref{Sec:assumptions} hold and let the parameters $\varepsilon,\zeta>0$ satisfy $\varepsilon\zeta^{-1}<1$.
    Then the invariant manifold $S_{\varepsilon,\zeta}$ is indeed the slow manifold of the system satisfying
    \begin{enumerate}
        \item[(i)] Regularity: The manifold $S_{\varepsilon,\zeta}$ is Lipschitz continuous and if the critical manifold $S_0$ is $C^k$-smooth, then, so is $S_{\varepsilon,\zeta}$.
        \item[(ii)] Distance: The slow manifold has a (Hausdorff semi-)distance of $\mathcal{O}(\varepsilon)$ to $S_0$ as $\varepsilon,\zeta\to 0$.
        \item[(iii)] Attraction: $S_{\varepsilon,\zeta}$ is a locally exponential attracting invariant manifold.
        \item[(iv)] Convergence of semi-flows: The semi-flow on the slow manifold $S_{\varepsilon,\zeta}$ converges to the semi-flow on the critical manifold $S_0$.
    \end{enumerate}
\end{theorem}
\begin{proof}
    The proof of the statements in the generalization of the Fenichel Theorem to the class of fast diffusion systems follow along the lines of the previous works \cite{hummel2022slow,kuehn2023fast}.
    Here, we focus on two aspects.
    
   The Lipschitz regularity of the slow manifold, i.e. the Lipschitz regularity of the fixed point mapping $h^{\varepsilon,\zeta}$ follows from the properties of the Lyapunov-Perron operator and the spectral gap condition \eqref{spectral gap condition}, where the condition $L_{\textnormal{spec}}<1$ is again key.
    
    Next, we compute the distance between the first component of the slow manifold $S_{\varepsilon,\zeta}$ and the critical manifold $S_0$.
    We estimate
    \begin{align*}
    \|h^{\varepsilon,\zeta}_{X_1}(v_0)-u_0\|_{X_1}&= \|h^{\varepsilon,\zeta}_{X_1}(v_0)\|_{X_1}\leq \varepsilon C \|v_0\|_{Y_1},
    \end{align*}
    which follows from the Lipschitz regularity of the slow manifold.
\end{proof}

\section{Slow Manifolds in the Fokker-Planck equation}\label{sec: slow manifold in FPE}

We continue with Step (S4) from Section \ref{sec: stochastic reduction} and recall the infinite system of coupled PDEs posed on the time and space domain $(0,T)\times (-R,R)$
\begin{align}\label{eq:infinite coefficient slow full}
\begin{split}
     \partial_t   a_0(t,y)&=C_0\frac{\sigma_2^2}{2}\partial_{y,y} a_0(t,y) -  \partial_y (G_{0}(y)  a_0(t,y)) -  \sum_{i \in \mathbb{N}\setminus \{0\}} \partial_y(G_{i}(y) a_{i}(t,y)), \end{split}\\ \label{eq:infinite coefficient fast full}
     \begin{split}   \partial_t a_j(t,y)  &= \frac{\sigma_2^2}{2}  \partial_{y,y} a_j(t,y) - \frac{\lambda_j}{\varepsilon}a_j(t,y)  - \partial_y \big( G_{0,j}(y) a_0(t,y)\big) - \sum_{i \in \mathbb{N}\setminus \{0\}}\partial_y \big( G_{i,j}(y) a_i(t,y)\big)\\
        &\quad + \tilde G_{0,j}(y)a_0(t,y)+ \sum_{i \in \mathbb{N}\setminus \{0\}}\tilde G_{i,j}(y) a_i(t,y),
    \end{split}
\end{align}
together with the following absorbing boundary and initial conditions
\begin{align}\label{eq: infinite coefficient bc}
    a_j(t,-R)=a_j(t,R)=0,\quad t\in (0,T) \textnormal{ and for all } j\in \mathbb{N},\\ \label{eq: infinite coefficient ic}
    a_j(0,y)=a_{j,0}(y), \quad y\in (-R,R) \textnormal{ and for all } j\in \mathbb{N},
\end{align}
where we used the notation
\begin{align}\label{eq:infinite definition G}
\begin{split}
    C_0= &\int_{\mathbb{R}} \phi^y_0(x)\,\textnormal{d} x,\quad G_k(y)= \int_{\mathbb{R}} g(x,y)  \phi_k^y(x)\,\textnormal{d} x,\\
    G_{k,j}(y)&= \int_{\mathbb{R}} g(x,y)  \phi_k^y(x)\phi_j^y(x)\,\textnormal{d} x  \quad \textnormal{ and }\quad  \tilde G_{k,j} (y)=  \int_{\mathbb{R}} g(x,y)  \phi_k^y(x) \partial_y \phi_j^y(x)\,\textnormal{d} x
\end{split}  
\end{align}
for all $j,k\in \mathbb{N}$. Instead of working with the full system, we work with a truncated one.
The reasons will become clear when we apply the framework for the existence of a slow manifold for a fast reaction system developed in the previous section.

Now, let $J\in \mathbb{N}$ be the level where we truncate the full system. That is we consider the following system of coupled PDEs
\begin{align}\label{eq:infinite coefficient slow truncated}
    \begin{split}
     \partial_t   a_0^J(t,y)&=C_0\frac{\sigma_2^2}{2}\partial_{y,y} a_0^J(t,y) -  \partial_y (G_{0}(y)  a_0^J(t,y)) -  \sum_{i =1}^J \partial_y(G_{i}(y) a_{i}^J(t,y)), \end{split}\\ \label{eq:infinite coefficient fast truncated}
     \begin{split}   \partial_t a_j^J(t,y)  &= \frac{\sigma_2^2}{2}  \partial_{y,y} a_j^J(t,y) - \frac{\lambda_j}{\varepsilon}a_j^J(t,y)  - \partial_y \big( G_{0,j}(y) a_0^J(t,y)\big) -  \sum_{i =1}^J\partial_y \big( G_{i,j}(y) a_i^J(t,y)\big)\\
        &\quad + \tilde G_{0,j}(y)a_0^J(t,y)+  \sum_{i =1}^J \tilde G_{i,j}(y) a_i^J(t,y),
    \end{split}
\end{align}
where $j=1,\dots,J$ together with the boundary and initial conditions \eqref{eq: infinite coefficient bc}-\eqref{eq: infinite coefficient ic}.
\begin{remark}
    Using the properties of the eigenfunctions $\phi_j^y(x)$ and the regularity assumption on the drift function $g$, see Assumption \ref{ass: solvability FPE}, we obtain that the functions $G_i,\, G_{i,j},\,\tilde G_{i,j}\in C^2([-R,R])$ for all $i,j\in \mathbb{N}$. 
    In particular, we have that these functions are bounded in $W^{2,\infty}(-R,R)$.
    For an example with explicit computations, we refer to section \ref{sec: examples}.
\end{remark}

Now, to relate the two systems we state the following approximation result.
\begin{prop}\label{prop: galerkin approx}
    Let $a_i^R(t,y)$ denote the difference between solutions $a_i(t,y)$ of the original system and $a_i^J(t,y)$ of the truncated system.
    Moreover, let the initial data satisfy $a_{j,0}(y)\in H^2(-R,R)$ for all $j\in \mathbb{N}$ and $\sum_{j\in \mathbb{N}}a_{j,0}(y)\in H^2(-R,R)$.
    Then, for all $t>0$ there exists a small constant $0<\delta \ll 1$ and a truncation level $J_0\in \mathbb{N}$ such that for all truncations at level $J\geq J_0$ it holds that
    \begin{align*}
        \|a_i^R(t)\|_{H^2(-R,R)}\leq \delta(J_0).
    \end{align*}
    In particular it holds that $\lim_{J_0\to \infty} \delta=0$.
\end{prop}

\begin{proof}
    From Lemma \ref{lemma: general existence coef} it follows that
    $$v^\varepsilon,\,w^\varepsilon \in C^1((0,T);L^2(\Omega))\cap C([0,T];H^2(\Omega)).$$
    Then, by definition of the coefficients $a_j$ we have that
    $$\sum_{i\in\mathbb{N}} a_j(t,y)\in C^1((0,T);L^2(-R,R))\cap C([0,T];H^2(-R,R)).$$
    In particular,
    $$\partial_t\sum_{i\in\mathbb{N}} a_j(t,y)\in C((0,T);L^2(-R,R)) $$
    and hence by equations \eqref{eq:infinite coefficient slow full}-\eqref{eq:infinite coefficient fast full} it follows that 
    \begin{align*}
        \sum_{i \in \mathbb{N}} \partial_y(G_{i}(y) a_{i}(t,y)),  \sum_{i \in \mathbb{N}} \partial_y \big( G_{i,j}(y) a_i(t,y)\big), \sum_{i \in \mathbb{N}} \tilde G_{i,j}(y) a_i(t,y)\in C((0,T);L^2(-R,R)),
    \end{align*}
    for all $j\in \mathbb{N}$.
    Thus, there exists a $J_0\in \mathbb{N}$ such that
    \begin{align*}
       \sup_{t\in (0,T)} \bigg\| \sum_{i =J_0+1}^\infty \partial_y(G_{i}(y) a_{i}(t,y))\bigg\|_{L^2}\leq \delta,
    \end{align*}
    and similar for the other two terms.
    Then, the statement follows by using an energy estimate on the difference between the solutions to the two systems.
    To be more precise we set $a_i^R:=a_i-a_i^J$ for all $i\in \mathbb{N}$ to be the difference between the full and the truncated system. 
    \begin{align*}
        \partial_t   a_0^R(t,y)&=C_0\frac{\sigma_2^2}{2}\partial_{y,y} a_0^R(t,y)  -  \sum_{i =0}^{J_0} \partial_y(G_{i}(y) a_{i}^R(t,y)) -\sum_{i =J_0+1}^{\infty} \partial_y(G_{i}(y) a_{i}^R(t,y)) , \\
    \partial_t a_j^R(t,y)  &= \frac{\sigma_2^2}{2}  \partial_{y,y} a_j^R(t,y) - \frac{\lambda_j}{\varepsilon}a_j^R(t,y)  -  \sum_{i =0}^{J_0}\partial_y \big( G_{i,j}(y) a_i^R(t,y)\big) - \sum_{i =J_0+1}^{\infty}\partial_y \big( G_{i,j}(y) a_i^R(t,y)\big)\\
        &\quad + \sum_{i =0}^{J_0} \tilde G_{i,j}(y) a_i^R(t,y)+  \sum_{i =J_0+1}^{\infty} \tilde G_{i,j}(y) a_i^R(t,y),~~\textnormal{for } j=1,\dots,J_0,\\
        \partial_t a_k^R(t,y)  &= \frac{\sigma_2^2}{2}  \partial_{y,y} a_k^R(t,y) - \frac{\lambda_j}{\varepsilon}a_k^R(t,y)  -  \sum_{i =0}^{J_0}\partial_y \big( G_{i,j}(y) a_i^R(t,y)\big) - \sum_{i =J_0+1}^{\infty}\partial_y \big( G_{i,j}(y) a_i^R(t,y)\big)\\
        &\quad + \sum_{i =0}^{J_0} \tilde G_{i,j}(y) a_i^R(t,y)+  \sum_{i =J_0+1}^{\infty} \tilde G_{i,j}(y) a_i^R(t,y),~~\textnormal{for } k=J_0+1,\dots,
    \end{align*}
    with Neumann boundary conditions and the following initial conditions
    \begin{align*}
        a_0^R(0)=0,\quad a_j^R(0)=0,~ \textnormal{for } j=1,\dots,J_0, \quad  a_k^R(0)=a_k(0) ~\textnormal{for } k=J_0+1,\dots.
    \end{align*}
    Then, testing each equation with the respective $a_i^R$ yields
    \begin{align*}
        \frac{1}{2}\frac{\textnormal{d}}{\textnormal{d}t} \|a_0^R\|^2 +C_0\frac{\sigma_2^2}{2}\|\partial_y a_0^R\|^2&\leq \|\partial_y a_0^R\|\sum_{i=1}^{J_0} \|G_i\|_{\infty}\| a_i^R\| +\delta\|a_0^R\| \\
        \frac{1}{2}\frac{\textnormal{d}}{\textnormal{d}t} \|a_j^R\|^2 +\frac{\sigma_2^2}{2}\|\partial_y a_j^R\|+ \frac{\lambda_j}{\varepsilon}\|a_j^R\|&\leq  \|\partial_y a_j^R\|\sum_{i=0, i\neq j}^{J_0} \|G_{i,j}\|_{\infty}\| a_i^R\|+ \|a_j^R\|\big(\sum_{i=0}^{J_0} \|\tilde G_{i,j}\|_{\infty}\|a_i^R\|+2\delta\big),
    \end{align*}
    for $j=1,\dots,J_0$. 
    Adding the estimates and using Young's inequality to control the terms on the right-hand side of the inequality yields
    \begin{align*}
        \frac{1}{2}\sum_{j=0}^{J_0} \frac{\textnormal{d}}{\textnormal{d}t} \|a_j^R\|^2 &\leq 2\delta \sum_{j=0}^{J_0} \|a_j^R\|
    \end{align*}
    and therefore, $\|a_j^R(t)\|\leq C(T)\delta^2$, which can be seen as a worst case estimate.
    We can use this estimate for the remainder of the sum which we estimate as follows
    \begin{align*}
        \frac{1}{2}\frac{\textnormal{d}}{\textnormal{d}t} \|a_k^R\|^2 +\frac{\sigma_2^2}{2}\|\partial_y a_k^R\|+ \frac{\lambda_j}{\varepsilon}\|a_k^R\|&\leq  \|\partial_y a_k^R\|\sum_{i=0, i\neq j}^{J_0} \|G_{i,j}\|_{\infty} C(T)\delta^2+ \|a_k^R\|\big(\sum_{i=0}^{J_0} \|\tilde G_{i,j}\|_{\infty}C(T)\delta^2+2\delta\big),
    \end{align*}
    for all $k=J_0+1\dots$.
    This expression can be further simplified as
    \begin{align*}
         \|a_k^R(t)\|^2 \leq \delta C(T) ~~\textnormal{ for all } ~k=J_0+1\dots,
    \end{align*}
    where the constant $C(T)$ depends on the final time $T$ and the nonlinear function $g$, but is independent of the truncation $J_0$.
\end{proof}

\begin{remark}
    We want to emphasize that the truncation method presented here differs from the Galerkin approximation of a fast-slow system in \cite{engel2021connecting,engel2022geometric}.
    As we apply the spectral decomposition in the stochastic reduction method only in one spatial direction we obtain an infinite system of coupled PDEs. 
    After truncating the system at a finite level $J_0\in \mathbb{N}$ we are left with a system of $J_0+1$ PDEs.
    Now, in order to obtain a geometric understanding of the fast-slow system we have to apply the results from section \ref{sec:abstract theory}, as the standard Fenichel theory only applies for systems of ODEs.
\end{remark}

\begin{remark}
    A different method to prove Proposition \ref{prop: galerkin approx} is to directly use estimates for the spectral Galerkin approximation, see \cite{funaro1991approximation,guo1999error,ma2005hermite} for details.
    However, the current proof presented above provides more insights in the structure of the approximate system, which will be useful when comparing the dynamics of the truncated system with the ones of the original system.
\end{remark}

Now, we can take the formal limit of the original system as $\varepsilon\to 0$, which is given by
\begin{align}
    \partial_t   a_0^0(t,y)&=C_0\frac{\sigma_2^2}{2}\partial_{y,y} a_0^0(t,y) -  \partial_y (G_{0}(y)  a_0^0(t,y)),\quad (y,t)\in (-R,R)\times (0,T],\\
    a_0^0(t,-R)&=a_0^0(t,R)=0,\quad t\in (0,T] ,\\
    a_0^0(0,y)&=a_{0}(y), \quad y\in (-R,R).
\end{align}

\begin{remark}
    Note, that we would obtain the same system by taking the limit as $\varepsilon \to 0$ in the truncated system.
\end{remark}

To show the existence of a slow manifold for the truncated system we need to check the assumptions of Section \ref{Sec:assumptions} and connect it with the setting of system \eqref{eq: infinite coefficient bc}-\eqref{eq:infinite coefficient fast truncated}.
Thus, we have that\\
\begin{itemize}
    \item The slow variable is given by $a_0$ and the fast variables are given by $a_1,\dots, a_{J_0}$.
    As underlying spaces we choose 
    $$Y=L^2(-R,R)\quad \textnormal{for the slow variable space and }   X^j= X=L^2(-R,R)~~\textnormal{for } j=1,\dots,J_0,$$
    for the fast variable space.
    \item The linear operator in the slow equation is given by 
    \begin{align*}
        B:L^2(-R,R)\supset H^2(-R,R)\cap H^1_0(-R,R)\to L^2(-R,R),\quad v\mapsto \partial_{y,y} v.
    \end{align*}
    The linear operator in the fast equation is given by
    \begin{align*}
        A^j:X\supset D(A^j)=X_1 \to X,\quad u_j\mapsto \varepsilon \frac{\sigma_2^2}{2}\partial_{y,y} u_j -\lambda_j u_j,~~\textnormal{for }j=1,\dots,J_0,
    \end{align*}
    where $D(A^j)=\{a^j\in X:\, a_j\in H^2(-R,R)\cap H^1_0(-R,R)\}$.
    Hence, both operators are closed and bounded.
    \item The spaces of Banach scales are given by $Y_\alpha=H^{2\alpha}(-R,R)\cap H^\alpha_0(-R,R)$ and $X^j_\alpha= \{a_j\in X:\, a_j\in H^{2\alpha}(-R,R)\cap H^\alpha_0(-R,R)\}$ for $\alpha\in [0,1]$.
    \item The Laplacian with zero Dirichlet boundary condition generates a bounded analytic semigroup on any space $H^{2\alpha}(-R,R)\cap H^\alpha_0(-R,R)$.
    Moreover, there are constants $C_A,C_B$ such that 
    \begin{align}
        \|\textnormal{e}^{A_j t}\|_{\mathcal{B}(X_{1/2},X_1)}\leq C_{A_j}\ t^{-1/2} \textnormal{e}^{\varepsilon \omega_{A,j}-\lambda_j t},\quad   \|\textnormal{e}^{B t}\|_{\mathcal{B}(X_{1/2},X_1)}\leq C_{B}\ t^{-1/2} \textnormal{e}^{\omega_B t},
    \end{align}
    where we can choose $\omega_{A,j}=\omega_B=0$ and we set $\gamma=\delta=1/2$.
    \item Next, we define the nonlinearities $\mathcal{F}_j: X_1 \times Y_1\to X_{1/2}$ for $j=1,\dots,J_0$  and $\mathcal{G}: X_1 \times Y_1\to Y_{1/2}$, where
    \begin{align*}
        \mathcal{G}(u,v)&= -\partial_y \big(G_{0}(y) v\big)-\sum_{i=1}^{J_0}\partial_y \big(G_{i}(y) u_i\big),
        \intertext{and}
        \mathcal{F}_j(u,v)&= - \partial_y \big( G_{0,j}(y) v\big) - \sum_{i \in 1}^{J_0}\partial_y \big( G_{i,j}(y)u_i\big)+ \tilde G_{0,j}(y)v + \sum_{i=1}^{J_0}\tilde G_{i,j}(y) u_i.   
    \end{align*}
    As the functions $\mathcal{F}_j$ and $\mathcal{G}$ are indeed linear it then follows from the assumptions on the drift function $g(x,y)$ in \eqref{eq: fast-slow SDE} that there exist constants $L_{\mathcal{F},j},L_\mathcal{G}>0$ such that
    \begin{align*}
         \|\textnormal{D} \mathcal{F}_j(u,v)\|_{\mathcal{B}(X_{1}\times Y_{1},X_{1/2})}&\leq L_{\mathcal{F}_j},\quad  \|\textnormal{D} \mathcal{G}(u,v)\|_{\mathcal{B}(X_{1}\times Y_{1},Y_{1/2})}\leq L_\mathcal{G}.
    \end{align*}
\end{itemize}
Thus, we can apply the convergence result of Lemma \ref{lemma: convergnce of sol} in the setting of the fast-slow coefficient system and obtain the following.
\begin{theorem}[Convergence of Solutions]\label{thm: converegnce of sol}
   Let $J_0\in \mathbb{N}$ be the truncation level. Then, under the assumptions of Section \ref{sec: problem set up} and assuming that the initial conditions satisfy 
   \begin{align*}
       a_0(0,y)=a_0^0(y,0)=a_0(y)\in H^2(-R,R),\quad \sum_{j=1}^{J_0} a_j(y,0)\in H^2(-R,R)
   \end{align*}
  the following estimate holds true for all $t>0$ 
   \begin{align}
       \begin{split}
           \sum_{j=1}^{J_0}\|a_j(t)\|_{H^2(-R,R)}+ \|a_0(t)-a^0_0(t)\|_{H^2(-R,R)}&\leq C\varepsilon^{1/2}\bigg( \sum_{j=1}^{J_0}\|a_j(0)\|_{H^2(-R,R)} +\|a_0(0)\|_{H^2(-R,R)}\bigg)\\
           &\quad + C \sum_{j=1}^{J_0} \textnormal{e}^{-\varepsilon^{-1}\lambda_j t}\|a_j(0)\|_{H^2(-R,R)}.
       \end{split}
   \end{align}
\end{theorem}

\begin{remark}
    Combining the results from Proposition \ref{prop: galerkin approx} and Theorem \ref{thm: converegnce of sol} also yields the convergence of solutions of the full system to solutions of the limit system.
\end{remark}

Now, it remains to check the assumptions for the existence of a slow manifold. Let $\zeta>0$ be the parameter that induces the splitting of the slow variable $a_0(t,y)$ satisfying $\varepsilon\zeta^{-1}<1$.
Then, we introduce the splitting as follows
\begin{align}
    Y_S^\zeta&:= \textnormal{span}\{[y\mapsto \sin(k\pi y/R)],\, k\in \mathbb{N}, k\leq k_0-1\},\\
    Y_F^\zeta&:=\textnormal{cl}_{L^2(-R,R)}\big( \textnormal{span}\{[y\mapsto \sin(k\pi y/R)],\, k\in \mathbb{N}, k\geq k_0\}\big),
\end{align}
where $-(k_0+1)^2<-\zeta^{-1}\min_{j}\lambda_j\leq -k_0^2$ for some $k_0\in \mathbb{N}$.
\\
The rest of the assumptions and estimates of the operator follows along the same lines as in \cite{hummel2022slow,kuehn2023fast}. 
The last assumption we need to check is then the spectral gap condition \eqref{spectral gap condition}, where we have
\begin{align}\label{eq: condition J spectral gap}
  L_{\textnormal{spec}}= \sum_{j=1}^{J_0} \frac{\varepsilon 2^{1/2}\Gamma(1/2) L_{\mathcal{F}_j}}{(2(\varepsilon\zeta^{-1} -1)(-\lambda_j)+ \varepsilon (N_S^\zeta+N_F^\zeta) )^{1/2}} + \frac{2^{1/2} L_\mathcal{G} \Gamma(1/2)}{(N_S^\zeta-N_F^\zeta)^{1/2}}+\frac{2L_\mathcal{G}  \Gamma(1/2)}{N_S^\zeta-N_F^\zeta},
\end{align}
for which we require $ L_{\textnormal{spec}}<1$.
By choosing
\begin{align*}
    N_S^\zeta= \zeta^{-1}\min_{j\in \mathbb{N}\setminus\{0\}}(\lambda_j)-(k_0-1)^2,\quad N_F^\zeta= \zeta^{-1}\min_{j\in \mathbb{N}\setminus\{0\}}(\lambda_j)-k_0^2+k_0-1.
\end{align*}
we have 
\begin{align*}
    N_S^\zeta-N_F^\zeta=k_0\geq \zeta^{-1/2}
\end{align*}
and thus we can control the last two terms of $ L_{\textnormal{spec}}$.
Moreover, the first term in the spectral gap estimate is controlled by fixing the truncation level $J_0$.



\begin{remark}
    Setting $\varepsilon=\zeta$ the above spectral gap condition reduces to the following problem
    \begin{align}\label{eq: condition J simplified}
      \sqrt{2\pi} \varepsilon^{1/2}\sum_{j=1}^{J_0} L_{\mathcal{F}_j}+  \sqrt{2\pi} \varepsilon^{1/2}L_\mathcal{G} +2\sqrt{\pi} \varepsilon L_\mathcal{G}=C \varepsilon^{1/2}<1,
    \end{align}
    which can be controlled for any finite $J_0$ by choosing $0<\varepsilon\ll 1$ sufficiently small.
    Here, we notice the importance of truncating the original system at a finite $J_0\in \mathbb{N}$ as we can control the spectral gap  for a fixed parameter $\varepsilon$ only for a finite number of Lipschitz constants $L_{\mathcal{F}_j}$.
    If one wants to consider both limits, i.e. the double singular limit  when $\varepsilon\to 0$ and $J_0\to \infty$ we refer the reader to \cite{kuehn2022general} for the details and intricacies that arise in this context.
\end{remark}

Hence, we obtain the following system with respect to the splitting into a fast and slow component
\begin{align}\label{eq:infinite coefficient splitting}
\begin{split}
     \partial_t   a_{0,S}(t,y)&=C_0\frac{\sigma_2^2}{2}\partial_{y,y}  a_{0,S}(t,y) -  \pr_{Y_S^\zeta}\partial_y (G_{0}(y)  a_0(t,y)) - \pr_{Y_S^\zeta} \sum_{i=1}^J \partial_y(G_{i}(y) a_{i}(t,y)), \end{split}\\ \begin{split}
      \partial_t   a_{0,F}(t,y)&=C_0\frac{\sigma_2^2}{2}\partial_{y,y}  a_{0,F}(t,y) -  \pr_{Y_F^\zeta}\partial_y (G_{0}(y)  a_0(t,y)) - \pr_{Y_F^\zeta} \sum_{i=1}^J \partial_y(G_{i}(y) a_{i}(t,y)), \end{split}\\ \label{eq:infinite coefficient slpitting end}
     \begin{split}   \partial_t a_j(t,y)  &= \frac{\sigma_2^2}{2}  \partial_{y,y} a_j(t,y) - \frac{\lambda_j}{\varepsilon}a_j(t,y)  - \partial_y \big( G_{0,j}(y) a_0(t,y)\big) - \sum_{i=1}^J\partial_y \big( G_{i,j}(y) a_i(t,y)\big)\\
        &\quad + \tilde G_{0,j}(y)a_0(t,y)+ \sum_{i=1}^J\tilde G_{i,j}(y) a_i(t,y),
    \end{split}
\end{align}
together with the initial conditions \eqref{eq: infinite coefficient ic} and boundary conditions \eqref{eq: infinite coefficient bc}.
Here $(t,y)\in (0,T]\times (-R,R)$ and $0<j\leq J$.
The well-posedness of this system follows from standard parabolic theory and the convergence of solutions to the approximate system \eqref{eq:infinite coefficient splitting}-\eqref{eq:infinite coefficient slpitting end} to solutions of the original system can be shown by using the standard arguments from Galerkin approximation results.



Then, for this system we can show the following result for the existence of a slow manifold.
\begin{theorem}[Existence of Slow Manifold]\label{thm: slow manifold fpe}
    Let $\varepsilon_0>0$ be given and let $\zeta>0$ be the parameter that induces the splitting of the slow variable $a_0(t,y)$ satisfying $0<\varepsilon_0 \zeta^{-1}\leq 1$.
    Moreover let $J_0\in \mathbb{N}$ be chosen such that the spectral gap condition \eqref{eq: condition J spectral gap} and the assumption of Theorem \ref{thm: converegnce of sol} hold.
    Then, for all $\varepsilon\in (0,\varepsilon_0] $ and all $J\leq J_0$ there exists a family of slow manifolds $S_{\varepsilon,\zeta}^J$ for the fast-slow system \eqref{eq:infinite coefficient splitting}-\eqref{eq:infinite coefficient slpitting end} together with the boundary and initial conditions \eqref{eq: infinite coefficient bc}-\eqref{eq: infinite coefficient ic}.
    The slow manifold is then given by
    \begin{align}
        S_{\varepsilon,\zeta}^J=\big\{ \big(h^{\varepsilon,\zeta,J}_1(a_{0,S}),\dots,h^{\varepsilon,\zeta,J}_J(a_{0,S}),h^{\varepsilon,\zeta,J}_F(a_{0,S}),a_{0,S}\big):\, a_{0,S}\in Y_1\cap Y_S^\zeta\big \} ,
    \end{align}
    where $h^{\varepsilon,\zeta,J}$ denotes the fixed point of the Lyapunov-Perron operator (see Lemma \ref{lemma: lyapunov perron}).
\end{theorem}
The dynamics on the slow manifold are then described by the following dimensionally reduced system for sufficiently small $\varepsilon$
\begin{align}\begin{split}
    \partial_t  a_{0,S}(t,y)&= C_0\frac{\sigma_2^2}{2}\partial_{y,y}  a_{0,S}(t,y) -  \pr_{Y_S^\zeta}\partial_y (G_{0}(y)  (a_{0,S}(t,y)+h^{\varepsilon,\zeta,J}_F(a_{0,S}(t,y)) ) \\
    &\quad -\pr_{Y_S^\zeta} \sum_{i\in 1} ^J\partial_y(G_{i}(y)h^{\varepsilon,\zeta,J}_i(a_{0,S}(t,y))), \end{split}\\
    a_{0,S}(t,-R)&= a_{0,S}(t,R)=0,\quad  a_{0,S}(0,y)= \pr_{Y_S^\zeta} a_{0,0}(y).
\end{align}

Now, we can combine the results to show that there exists a slow manifold for any $\varepsilon\in (0,\varepsilon_0]$ such that the error in the truncation is also bounded by $\varepsilon$.
\begin{corollary}
    There exists $\varepsilon\in (0,\varepsilon_0]$ and a truncation $J=J(\varepsilon)$ such that the result of Theorem \ref{thm: slow manifold fpe} holds for $\zeta=\varepsilon$, i.e. there exists a slow manifold $S_{\varepsilon}^{J(\varepsilon)}$, and moreover the difference between the original and truncated system satisfies
    \begin{align*}
        \|a_i(t)-a^{J(\varepsilon)}_i(t)\|_{H^2(-R,R)} \leq C \varepsilon.
    \end{align*}
\end{corollary}

The question now is, how do we relate this result to the stochastic reduction method?\\
First, we multiply each coefficient function $a_j$ with the corresponding orthogonal function $\phi_j$ for all $j=0,\dots,J$ and obtain the functions 
\begin{align*}
    v_S^\varepsilon(t,x,y)= a_{0,S}(t,y)\phi_0^y(x),\quad  v_F^\varepsilon(t,x,y)= a_{0,F}(t,y)\phi_0^y(x),\quad w^{\varepsilon,J}(t,x,y)=\sum_{j=1}^J a_j(t,y)\phi_j^y(x),
\end{align*}
where
\begin{align*}
    &v_S^\varepsilon(t,x,y)\in C^0\big([0,T];\mathcal{H}^2_{0,0}(\mathbb{R})\times (H^2(-R,R)\cap H^1_0(-R,R)\cap Y_S^\zeta(-R,R))\big)\cap C^1\big((0,T];\mathcal{L}^2_{0,0}(\mathbb{R})\times Y_S^\zeta(-R,R)\big),\\
    &v_F^\varepsilon(t,x,y)\in C^0\big([0,T];\mathcal{H}^2_{0,0}(\mathbb{R})\times (H^2(-R,R)\cap H^1_0(-R,R)\cap Y_F^\zeta(-R,R))\big)\cap C^1\big((0,T];\mathcal{L}^2_{0,0}(\mathbb{R})\times Y_F^\zeta(-R,R)\big),\\
    &w^{\varepsilon,J}(t,x,y)\in C^0\big([0,T];\mathcal{H}^2_{1,J}(\mathbb{R})\times (H^2(-R,R)\cap H^1_0(-R,R))\big)\cap C^1\big((0,T];\mathcal{L}^2_{1,J}(\mathbb{R})\times L^2(-R,R)\big),
\end{align*}
where $\mathcal{H}^2_{i,j}(\mathbb{R})=\pr_{\textnormal{span}\{\phi_i,\dots,\phi_j\}}H^2(\mathbb{R})$ and $\mathcal{L}^2_{i,j}(\mathbb{R})=\pr_{\textnormal{span}\{\phi_i,\dots,\phi_j\}}L^2(\mathbb{R})$.\\

Due to the special structure of the original system and the properties of the stochastic reduction method these functions then satisfy
\begin{align}\label{eq: retransformed v}
    \partial_t(v_F^\varepsilon+v_S^\varepsilon)&= \PP \mathfrak{L}_2 (v_F^\varepsilon+v_S^\varepsilon+w^{\varepsilon,J}),\\ \label{eq: retransformed w}
    \partial_t w^{\varepsilon,J}&= \frac{1}{\varepsilon}\mathfrak{L}_1 w^{\varepsilon,J}+ \QQ \mathfrak{L}_2(v_F^\varepsilon+v_S^\varepsilon+w^{\varepsilon,J}).
\end{align}
The original system expressed in these functions takes the form
\begin{align*}
 \partial_t (v_F^\varepsilon+v_S^\varepsilon)&= \PP \mathfrak{L}_2 (v_F^\varepsilon+v_S^\varepsilon+w^{\varepsilon,J})+R^J_1(w^{\varepsilon,R}),\\ 
    \partial_t w^{\varepsilon,J}&= \frac{1}{\varepsilon}\mathfrak{L}_1 w^{\varepsilon,J}+ \QQ \mathfrak{L}_2(v_F^\varepsilon+v_S^\varepsilon+w^{\varepsilon,J})+R^J_2(w^{\varepsilon,R}),
\end{align*}
where $R^J_1,R^J_2$ are small error terms depending on the reminder term $w^{\varepsilon,R}= \sum_{j=J+1}^\infty a_j(t,y)\phi_j^y(x)$.
By Proposition \ref{prop: galerkin approx} it follows that the reminder term satisfies $\|R^J_{1,2}(w^{\varepsilon,R})\|_{L^2(-R,R)}=0$ as the truncation $J\to \infty$. 

The slow manifold for system \eqref{eq: retransformed v}-\eqref{eq: retransformed w} enhanced with the proper boundary and initial condition then reads
\begin{align}
    \mathsf{S}^J_{\varepsilon,\zeta}=\bigg\{ \big(\sum_{j=1}^J h^{\varepsilon,\zeta,J}_j(v_S^\varepsilon) ,h^{\varepsilon,\zeta,J}_F(v_S^\varepsilon),v_S^\varepsilon\big):\, v_S^\varepsilon\in \mathcal{H}^2_{0,0}(\mathbb{R})\times (H^2(-R,R)\cap H^1_0(-R,R)\cap Y_S^\zeta(-R,R))\bigg \}.
\end{align}
Then, recalling the definition of the operators $\PP$ and $\QQ$ and their relation with $v$ and $w$ we obtain a system for the modified Fokker-Planck equation
\begin{align}
    \partial_t \varrho^\varepsilon &= \frac{1}{\varepsilon}\mathfrak{L}_1 \varrho^\varepsilon+ \mathfrak{L}_2 \varrho^\varepsilon,
\end{align}
where $\varrho^\varepsilon= v_F^\varepsilon+v_S^\varepsilon+w^{\varepsilon,J} $ and
\begin{align*}
    \varrho^\varepsilon(t,x,y)\in  C^0\big([0,T];\mathcal{H}^2_{0,J}(\mathbb{R})\times (H^2(-R,R)\cap H^1_0(-R,R))\big)\cap C^1\big((0,T];\mathcal{L}^2_{0,J}(\mathbb{R})\times L^2(-R,R)\big)
\end{align*}
Again, we can compare the approximation to the original Fokker-Planck equation, now written in the new function space,
 \begin{align*}
      \partial_t \varrho^\varepsilon &= \frac{1}{\varepsilon}\mathfrak{L}_1 \varrho^\varepsilon+ \mathfrak{L}_2 \varrho^\varepsilon+ R(\rho^\varepsilon-\varrho^\varepsilon),
 \end{align*}
 where the error term can be bounded in terms of $\varepsilon$ for $\rho^\varepsilon-\varrho^\varepsilon$ sufficiently small.
Furthermore, we can recover the slow manifold as in the previous step
\begin{align}\label{eq: slow manifold FPE}
    \mathcal{S}^J_{\varepsilon,\zeta}=\big\{  \big(h^{\varepsilon,\zeta,J}_{F,\varrho}(\varrho^\varepsilon_S),\varrho^\varepsilon_S\big):\, \varrho^\varepsilon_S\in \mathcal{H}^2_{0,0}(\mathbb{R})\times (H^2(-R,R)\cap H^1_0(-R,R)\cap Y_S^\zeta(-R,R)) \big \},
\end{align}
where $\varrho^\varepsilon_S=\pr_{\textnormal{span}(\phi_0)}\pr_{Y_S^\zeta} \varrho^\varepsilon$ and the slow dynamics on the critical manifold are given by
\begin{align}\label{eq: reduced dynamics FPE}
    \partial_t\varrho^\varepsilon_S&= \mathfrak{L}_2 (\varrho^\varepsilon_S+ h^{\varepsilon,\zeta,J}_{F,\varrho}(\varrho^\varepsilon_S)).
\end{align}
We recall that the limit system as $\varepsilon\to 0$ for the fast-slow system \eqref{eq: v equation}-\eqref{eq: w equation} has the form
\begin{align*}
    \partial_t v^0&= \PP \mathfrak{L}_2 v^0,\quad     w^0=0.
\end{align*}
Hence, the limit of the Fokker-Planck equation is given by
\begin{align*}
    \partial_t \rho^0= \mathfrak{L}_2 \rho^0,
\end{align*}
where $\rho^0$ is defined on the critical manifold 
\begin{align}
   \mathcal{S}_0=\{(0,\rho):\, \rho \in \mathcal{H}^2_{0,0}(\mathbb{R})\times (H^2(-R,R)\cap H^1_0(-R,R)\}. 
\end{align}

Hence, we can apply the result of Theorem \ref{thm: slow manifold fpe} on the generalization of the slow manifold theory to fast diffusion systems to the reduced Fokker-Planck equation \eqref{eq: slow manifold FPE}-\eqref{eq: reduced dynamics FPE} and obtain that the slow manifold $\mathcal{S}_{\varepsilon,\zeta}^J$ has the same regularity and attractiveness properties as the critical manifold $\mathcal{S}_0$.   
Moreover, the distance between these two objects is of order $\mathcal{O}(\varepsilon)$, for sufficiently small parameters $\varepsilon,\zeta$.\\

To conclude this section, we have shown how the stochastic reduction method together with a Galerkin-type approximation and geometric singular perturbation theory can prove the existence of a slow manifold for an approximate fast-slow Fokker-Planck equation, with error in the approximation of order $\mathcal{O}(\varepsilon)$ to the original system stemming from a fast-slow stochastic differential equation.

\section{Applications}\label{sec: examples}

In this section we apply the results for the existence of an approximate slow manifold for a fast-slow Fokker-Planck equation from the previous section to the following linear system
\begin{align}\label{eq:linear system}
    \begin{split}
        \textnormal{d} x&= \frac{1}{\varepsilon}(y-x)\textnormal{d}t + \frac{\sqrt{2}}{\sqrt{\varepsilon}}\textnormal{d}W_1,\\
        \textnormal{d} y&= -x \textnormal{d}t + \sqrt{2}\textnormal{d}W_2.
    \end{split}
\end{align}
Note that in the deterministic setting the critical and slow manifold can be computed explicitly.

The Fokker-Planck equation for the probability density $\rho$ corresponding to system \eqref{eq:linear system} takes the form 
\begin{align}\label{eq: FPE linear system}
    \partial_t \rho^\varepsilon=\frac{1}{\varepsilon} \partial_x\big(x-y+ \partial_x\big) \rho^\varepsilon+ \partial_y \big(x+\partial_y\big)\rho^\varepsilon,
\end{align}
where the equation is posed on the infinite strip $\Omega= \mathbb{R}\times (-R,R)$.
The stationary distribution/critical manifold can be computed explicitly and is given by
\begin{align*}
    p_s(x)= \frac{1}{\sqrt{2\pi}}\textnormal{e}^{-1/2(y-x)^2}
\end{align*}
and we obtain that in the limit $\varepsilon \to 0$ the reduced system for the marginal density $\bar \rho$ is given by 
\begin{align}
  \partial_t  \bar \rho (t,y)= \partial_y\big(y + \partial_y\big)\bar \rho(t,y),
\end{align}
where $\bar \rho(t,y)= \int_{\mathbb{R}} \rho(t,x,y)\,\textnormal{d}x$.
\begin{remark}
    This equation can also be obtained by taking the formal limit in the SDE and then writing down the Fokker-Planck equation. 
\end{remark}

Next, we consider the equation $\mathfrak{L}_1 \phi(x)= -\lambda \phi(x)$ in the setting of the Sturm-Liouville theory.
As it turns out, for this example, we can compute the eigenvalues and corresponding eigenfunctions explicitly.
We obtain that
\begin{align*}
    \lambda=n=0,1,2,\dots,~\text{ where } n\in \mathbb{N},\quad \phi_n^y(x)= \textnormal{e}^{-1/2(x-y)^2} H_n\bigg(\frac{x-y}{\sqrt{2}}\bigg),
\end{align*}
where $H_n(z)$ denotes the $n$-th Hermite polynomial \cite{andrews1999special}. 
Using the properties of the Hermite polynomials we can write
\begin{align*}
    g(x,y)=-x=-\frac{1}{\sqrt{2}} H_1\bigg(\frac{x-y}{\sqrt{2}}\bigg) -y H_0\bigg(\frac{x-y}{\sqrt{2}}\bigg).
\end{align*}
We start with computing the functions $C_0$, $G_j(y)$, $G_{k,j}(y)$ and $\tilde G_{k,j}(y)$ for $k,j\in \mathbb{N}_0$.
\begin{align*}
    C_0&= \int_{\mathbb{R}} \phi_0^y(x)\,\textnormal{d} x = \sqrt{2\pi},\\
    G_j&= \int_{\mathbb{R}} g(x,y)\phi_j^y(x)\,\textnormal{d} x= \int_{\mathbb{R}}-\frac{1}{\sqrt{2}} H_1\bigg(\frac{x-y}{\sqrt{2}}\bigg) -y H_0\bigg(\frac{x-y}{\sqrt{2}}\bigg)\phi_j^y(x)\,\textnormal{d} x\\
    &= -y \sqrt{2\pi} \delta_{0,j}-2\sqrt{\pi}\delta_{1,j}
\end{align*}
Using the recurrence relation for Hermite polynomials $H_{n+1}(z)=2z H_n(z)-2nH_{n-1}(z)$
we have
\begin{align*}
   G_{k,j}(y)&= -y \int_{\mathbb{R}} H_k(\frac{x-y}{\sqrt{2}})H_j(\frac{x-y}{\sqrt{2}})\textnormal{e}^{-(x-y)^2} \,\textnormal{d} x - \frac{1}{\sqrt{2}}\int_{\mathbb{R}}  H_1(\frac{x-y}{\sqrt{2}}) H_k(\frac{x-y}{\sqrt{2}})H_j(\frac{x-y}{\sqrt{2}}) \textnormal{e}^{-(x-y)^2}\,\textnormal{d} x\\
   &=- y\sqrt{2}\int_{\mathbb{R}} H_k(z)H_j(z)\textnormal{e}^{-2z^2}\, \textnormal{d}z -\int_{\mathbb{R}}  H_1(z) H_k(z)H_j(z) \textnormal{e}^{-2z^2}\,\textnormal{d}z\\
    &= -y\sqrt{2}\int_{\mathbb{R}}  H_k(z)H_j(z)\textnormal{e}^{-2z^2}\,\textnormal{d}z  -\int_{\mathbb{R}}   \big( H_{k+1}(z)+2kH_{k-1}(z)\big) H_j(z) \textnormal{e}^{-2z^2}\,\textnormal{d}z\\
   &= -y\sqrt{2}\delta_{k,j} C_{k,j} -\delta_{k+1,j} C_{k+1,j} - 2k\delta_{k-1,j} C_{k-1,j},
\end{align*}
where the constants $C_{j,j}$ arise as the eigenfunctions $\phi_j^y(x)$ have not yet been normalised. 
Thus, we have
\begin{align*}
  C_{0,0}= \sqrt{\pi}, \quad   C_{j,j}= \int_{\mathbb{R}}  \phi_j^y(x)\phi_j^y(x)\,\textnormal{d}x= \sqrt{\frac{\pi}{2}}(2j-1)!!,~j\geq 1,
\end{align*}
where $(2j-1)!!$ denotes the odd factorial.
For the special case $G_{0,j}(y)$ we have $G_{0,j}(y)= -\sqrt{\frac{\pi}{2}}\delta_{1,j}$.

Using that $\partial_z\big( H_n(z) \textnormal{e}^{-z^2}\big)=-H_{n+1}(z) \textnormal{e}^{-z^2}$ we have
\begin{align*}
    \tilde G_{k,j} (y)&=  \int_{\mathbb{R}}  g(x,y) H_k(\frac{x-y}{\sqrt{2}})  \textnormal{e}^{-1/2(x-y)^2} \partial_y\bigg(H_{j}(\frac{x-y}{\sqrt{2}})  \textnormal{e}^{-1/2(x-y)^2}\bigg)\,\textnormal{d}x\\
    &= \sqrt{2}\int_{\mathbb{R}} \big(-y H_0(z) -\frac{1}{\sqrt{2}}H_1(z)\big) H_k(z) \textnormal{e}^{-z^2} \partial_z\bigg(H_{j}(z)  \textnormal{e}^{-z^2}\bigg) \frac{-1}{\sqrt{2}}\,\textnormal{d}z\\
    &=\int_{\mathbb{R}} \big(-y H_0(z) -\frac{1}{\sqrt{2}}H_1(z)\big) H_k(z) \textnormal{e}^{-z^2} H_{j+1}(z)  \textnormal{e}^{-z^2}\,\textnormal{d}z\\
    &= -y \int_{\mathbb{R}}  H_k(z) H_{j+1}(z)  \textnormal{e}^{-2z^2}\,\textnormal{d}z-\frac{1}{\sqrt{2}} \int_{\mathbb{R}}  \big( H_{k+1}(z)+ 2kH_{k-1}(z)\big) H_{j+1}(z)  \textnormal{e}^{-2z^2}\,\textnormal{d}z\\
    &= -y \delta_{k,j+1} C_{k,j+1} -\frac{1}{\sqrt{2}}\big(\delta_{k+1,j+1}C_{k+1,j+1}+2k\delta_{k-1,j+1}C_{k-1,j+1}\big).
\end{align*}

Then, for the slow component we have 
\begin{align} \label{eq: example slow}
      \partial_t   a_0(t,y)&=  \partial_{y,y} a_0(t,y)+  \partial_y (y  a_0(t,y)) +\sqrt{2}\partial_y a_{1}(t,y).
\end{align}
Similar, for the fast component we have
\begin{align*}
    C_{j,j}\partial_t a_j(t,y)  &=  C_{j,j}\partial_{y,y} a_j(t,y)-   C_{j,j}\frac{\lambda_j}{\varepsilon}a_j(t,y)  - \partial_y \big( G_{0,j}(y) a_0(t,y)\big) - \sum_{i \in \mathbb{N}\setminus \{0\}}\partial_y \big( G_{i,j}(y) a_i(t,y)\big)\\
        &\quad + \tilde G_{0,j}(y)a_0(t,y)+ \sum_{i \in \mathbb{N}\setminus \{0\}}\tilde G_{i,j}(y) a_i(t,y)
\end{align*}
which can be simplified using the previous results
\begin{align}\label{eq: example fast}
\begin{split}
    \partial_t a_j(t,y) &= \partial_{y,y} a_j(t,y) - \frac{j}{\varepsilon}a_j(t,y)  +\sqrt{2}  \partial_y\big( y a_j(t,y)\big ) -(2j+1) a_j(t,y)\\
    &\quad+\delta_{1,j}\partial_y a_1(t,y)+ \partial_y a_{j-1}(t,y) +2(j+1)\partial_y a_{j+1}(t,y)\\
    &\quad -y (2j+1) a_{j+1}(t,y) -2(j+2)(2j+1)a_{j+2}(t,y).
    \end{split}
\end{align}
To close the system we recall the following boundary and initial conditions 
\begin{align}
    a_j(t,-R)&=0=a_j(t,R),\quad \textnormal{for all } t>0,\, j\in \mathbb{N},\\
    a_j(0,y)&=a_{j,0}(y),\quad \textnormal{for all } y\in (-R,R) \, j\in \mathbb{N},
\end{align}
where $\sum_{j\in \mathbb{N}} a_{0,j}(y)\phi_j^y(x)=\rho_0(x,y)$.

In the next step we apply the results on the existence of a slow manifold of the previous section. 
To this end we have to truncate the system at a finite $J\in \mathbb{N}$. 
Using that the Lipschitz constant of the nonlinearities in equation \eqref{eq: example fast} satisfies $L_{\mathcal{F}_j}=Cj^2$ and the relation between the truncation $J$ and $\varepsilon$, $\varepsilon^{1/2} \sum_{j=0}^JC j^2<1$, we obtain that $\varepsilon^{1/2}J^3<1$, i.e, the truncation at $J$ has to be of order $\mathcal{O}(\varepsilon^{-1/6})$ for the spectral gap condition \eqref{spectral gap condition} to hold.
Then, we can apply Theorem \ref{thm: slow manifold fpe} to obtain
\begin{align*}
    S_{\varepsilon,\zeta}^J=\big\{ \big(h^{\varepsilon,\zeta,J}_1(a_{0,S}),\dots,h^{\varepsilon,\zeta,J}_J(a_{0,S}),h^{\varepsilon,\zeta,J}_F(a_{0,S}),a_{0,S}\big):\, a_{0,S}\in Y_1\cap Y_S^\zeta\big \},
\end{align*}
where $h^{\varepsilon,\zeta}$ is the fixed point of the Lyapunov-Perron operator (see Lemma \ref{lemma: lyapunov perron}).
Following the steps from the previous section we derive the approximate slow manifold of system \eqref{eq: FPE linear system} given by
\begin{align*}
    \mathcal{S}_{\varepsilon,\zeta}^J=\big\{  \big(h^{\varepsilon,\zeta,J}_{F,\varrho}(\varrho^\varepsilon_S),\varrho^\varepsilon_S\big):\, \varrho^\varepsilon_S\in \mathcal{H}^2_{0,0}(\mathbb{R})\times (H^2(-R,R)\cap H^1_0(-R,R)\cap Y_S^\zeta(-R,R)) \big \}
\end{align*}
together with the reduced Fokker-Planck equation 
\begin{align*}
   \partial_t\varrho^\varepsilon_S&= \partial_y^2(\varrho^\varepsilon_S+ h^{\varepsilon,\zeta,J}_{F,\varrho}(\varrho^\varepsilon_S)) + \partial_y\big(x(\varrho^\varepsilon_S+ h^{\varepsilon,\zeta,J}_{F,\varrho}(\varrho^\varepsilon_S))\big) .
\end{align*}
 

\section{Discussion}\label{sec: discussion}

We conclude this article by a discussion of the results and put them into the broader context of the approaches and results mentioned in the introduction. \\

The first observation we make is that the approach presented in this paper can be extended to Fokker-Planck equations derived from fast-slow stochastic differential equations with non-Gaussian noise.
One example is the case of stable L\'evy noise \cite{zolotarev1986one,albeverio2000invariant}.
Consider the following SDE
\begin{align*}
    \textnormal{d}x= f(t,x)\textnormal{d}t +F(t,x)\textnormal{d}L,
\end{align*}
where the the stable L\'evy noise is given by a L\'evy stable distribution $S_k=S_k(\alpha,\beta,\gamma,p)$, depending on four parameters: the stability index $\alpha\in (0,2]$, the skewness parameter $\beta\in [-1,1]$, the scale parameter $\lambda\in (0,\infty)$ and a location parameter $p\in (-\infty,\infty)$.
Excluding the singular case $\alpha=1$ and $\beta\neq 0$ and assuming $p=0$ we obtain
\begin{align*}
    S_k(\alpha,\beta,\gamma)= \exp\big[-\gamma |k|^\alpha\big(1+i\beta \textnormal{sgn}(k)\tan(\pi \alpha/2)\big)\big].
\end{align*}
Then, the corresponding generalized Fokker-Planck equation is given by
\begin{align*}
  \frac{\partial}{\partial t} \rho(x,t)&= -\frac{\partial}{\partial x} \big( f(x,t)\rho (x,t)\big) +\gamma\frac{\partial^\alpha}{\partial|x|^\alpha}\big( F^\alpha(x,t)\rho(x,t)\big) +\gamma \beta \tan(\pi\alpha/2) \frac{\partial}{\partial x}\frac{\partial^{\alpha-1}}{\partial |x|^{\alpha-1}}\big(F^\alpha(x,t)\rho(x,t)\big).
\end{align*}
For details of the derivation we refer to \cite{schertzer2001fractional,denisov2009generalized,gao2016fokker}.
It is obvious that for $\alpha=2$ this equation takes the form of the standard Fokker-Planck equation.
Note, that for $\alpha>1$ the leading order derivative is given by a fractional Laplacian, for which the assumptions of the semigroup theory of section \ref{sec:abstract theory} still hold \cite{achleitner2024fractional}.
In the context of fast-slow systems some work was done by Duan et al. \cite{xu2011averaging,huang2022homogenization}.
However, with the novel approach presented in this work we can establish the existence of a slow manifold also for the case of stable L\'evy noise and thus extend the previous results.

More generally speaking, whenever the operators in a generalized Fokker-Planck equation, stemming from a non-Gaussian noise SDE, can be treated as the generators of a semigroup our approach can be applied.
Other examples besides the stable L\'evy noise are the Poisson white noise or compound noise as a combination of the three types \cite{denisov2009generalized}.\\

A second observation we make is that although the starting point of the stochastic reduction method is a fast-slow SDE, we obtain a slow manifold with corresponding slow dynamics for the Fokker-Planck equation derived from the original fast-slow SDE. 
The question now is if and how it is possible to obtain a SDE for the reduced Fokker-Planck equation.
Using different methods such as a probabilistic interpretation of the Fokker-Planck equation or sampling methods a partial answer to this problem was given in  \cite{barbu2018probabilistic,barbu2020nonlinear} and \cite{boffi2023probability,albergo2023stochastic}.\\

We remark that it is also possible to consider the problem with measurable initial data such as $\rho_0\in L^1(\Omega)$.
The elemental results for the $L^1$-semigroup theory can be found in \cite{arendt2002semigroups,angiuli2010analytic} and the references therein.
Again, with the underlying structure being well-defined, we can transfer the results of Section \ref{sec: slow manifold in FPE} also to the case of $L^1$-data.\\


Lastly, we want to briefly connect the slow manifold stochastic reduction method for a fast-slow Fokker-Planck equation with the stochastic averaging approach as presented in \cite{pavliotis2008multiscale} and \cite{walter2006conditional}.
For a fast-slow system of the form
\begin{align*}
     \textnormal{d}x&= \frac{1}{\varepsilon} f(x,y) \textnormal{d}t +\frac{\sigma}{\sqrt{\varepsilon}}\textnormal{d}W_t,\quad x\in \mathbb{R},\\
    \textnormal{d}y&=  g(x,y) \textnormal{d}t +\sigma'\textnormal{d}W_t,\quad y\in \mathbb{R},
\end{align*}
the stochastic averaging gives the reduced limit equation 
\begin{align*}
    \textnormal{d}y^0&= \bigg(\int_{\mathbb{R}}  g(x,y^0)p_s(x)\, \textnormal{d}x\bigg) \textnormal{d}t +\sigma'\textnormal{d}W_t,\quad y^0\in \mathbb{R}^m,
\end{align*}
where $p_s(x)$ is the invariant density (see equations \eqref{eq: stationary distr} and \eqref{eq: stationary distr 2}).
The approach and results presented in this paper, however, go beyond the stochastic averaging principle by introducing the concept of a slow manifold for the Fokker-Planck equation and obtaining also a reduced equation for the perturbed system, which is of order $\mathcal{O}(\varepsilon)$ close to the limit equation (see results in section \ref{sec: slow manifold in FPE}).
In a future work we want to connect our approach, i.e. the slow manifold for the fast-slow Fokker-Planck equation, to the slow manifolds for fast-slow random dynamical systems introduced in \cite{schmalfuss2008invariant} and study the relation between the two objects.\\

\section*{Acknowledgements}
CK and JES would like to thank the DFG for partial support via grant 456754695. CK would like to thank the VolkswagenStiftung for support via a Lichtenberg Professorship.
Moreover, the authors would like to thank Giacomo Landi for the helpful feedback and suggestions.

\bibliographystyle{siam}
\bibliography{lit}

\begin{thebibliography}{10}

\bibitem{achleitner2024fractional}
{\sc F.~Achleitner, G.~Akagi, C.~Kuehn, J.~M. Melenk, J.~D. Rademacher,
  C.~Soresina, and J.~Yang}, {\em Fractional dissipative pdes}, in Fractional
  Dispersive Models and Applications: Recent Developments and Future
  Perspectives, Springer, 2024, pp.~53--122.

\bibitem{albergo2023stochastic}
{\sc M.~S. Albergo, N.~M. Boffi, and E.~Vanden-Eijnden}, {\em Stochastic
  interpolants: A unifying framework for flows and diffusions}, arXiv preprint
  arXiv:2303.08797,  (2023).

\bibitem{albeverio2000invariant}
{\sc S.~Albeverio, B.~R{\"u}diger, and J.-L. Wu}, {\em Invariant measures and
  symmetry property of l{\'e}vy type operators}, Potential Analysis, 13 (2000),
  pp.~147--168.

\bibitem{amann1995linear}
{\sc H.~Amann}, {\em Linear and quasilinear parabolic problems, Vol I},
  Birkhäuser, Basel, 1995.

\bibitem{andrews1999special}
{\sc G.~E. Andrews, R.~Askey, R.~Roy, R.~Roy, and R.~Askey}, {\em Special
  functions}, vol.~71, Cambridge University Press, Cambridge, 1999.

\bibitem{angiuli2010analytic}
{\sc L.~Angiuli, D.~Pallara, and F.~Paronetto}, {\em Analytic semigroups
  generated in $\textnormal{L}^1({\Omega})$ by second order elliptic operators
  via duality methods}, in Semigroup Forum, vol.~80, Springer, 2010,
  pp.~255--271.

\bibitem{arendt2002semigroups}
{\sc W.~Arendt}, {\em Semigroups and evolution equations: functional calculus,
  regularity and kernel estimates}, in Handbook of differential equations:
  Evolutionary equations, vol.~1, Elsevier, 2002, pp.~1--85.

\bibitem{barbu2018probabilistic}
{\sc V.~Barbu and M.~R{\"o}ckner}, {\em Probabilistic representation for
  solutions to nonlinear fokker--planck equations}, SIAM Journal on
  Mathematical Analysis, 50 (2018), pp.~4246--4260.

\bibitem{barbu2020nonlinear}
\leavevmode\vrule height 2pt depth -1.6pt width 23pt, {\em From nonlinear
  fokker--planck equations to solutions of distribution dependent sde}, The
  Annals of Probability, 48 (2020), pp.~1902--1920.

\bibitem{berglund2003geometric}
{\sc N.~Berglund and B.~Gentz}, {\em Geometric singular perturbation theory for
  stochastic differential equations}, Journal of Differential Equations, 191
  (2003), pp.~1--54.

\bibitem{berglund2006noise}
\leavevmode\vrule height 2pt depth -1.6pt width 23pt, {\em Noise-induced
  phenomena in slow-fast dynamical systems: a sample-paths approach}, Springer
  Science \& Business Media, 2006.

\bibitem{boffi2023probability}
{\sc N.~M. Boffi and E.~Vanden-Eijnden}, {\em Probability flow solution of the
  fokker--planck equation}, Machine Learning: Science and Technology, 4 (2023),
  p.~035012.

\bibitem{bogachev2008parabolic}
{\sc V.~I. Bogachev, G.~Da~Prato, and M.~R{\"o}ckner}, {\em On parabolic
  equations for measures}, Communications in Partial Differential Equations, 33
  (2008), pp.~397--418.

\bibitem{bogachev2022fokker}
{\sc V.~I. Bogachev, N.~V. Krylov, M.~R{\"o}ckner, and S.~V. Shaposhnikov},
  {\em Fokker--Planck--Kolmogorov Equations}, vol.~207, American Mathematical
  Society, 2022.

\bibitem{bris2008existence}
{\sc C.~L. Bris and P.-L. Lions}, {\em Existence and uniqueness of solutions to
  fokker--planck type equations with irregular coefficients}, Communications in
  Partial Differential Equations, 33 (2008), pp.~1272--1317.

\bibitem{cerrai2001second}
{\sc S.~Cerrai}, {\em Second order PDE’s in finite and infinite dimension: a
  probabilistic approach}, Springer, 2001.

\bibitem{denisov2009generalized}
{\sc S.~I. Denisov, W.~Horsthemke, and P.~H{\"a}nggi}, {\em Generalized
  fokker-planck equation: Derivation and exact solutions}, The European
  Physical Journal B, 68 (2009), pp.~567--575.

\bibitem{duan2013slow}
{\sc J.~Duan, H.~Fu, and X.~Liu}, {\em Slow manifolds for multi-time-scale
  stochastic evolutionary systems}, Communications in Mathematical Sciences, 11
  (2013), pp.~141--162.

\bibitem{engel2021connecting}
{\sc M.~Engel, F.~Hummel, and C.~Kuehn}, {\em Connecting a direct and a
  galerkin approach to slow manifolds in infinite dimensions}, Proceedings of
  the American Mathematical Society, Series B, 8 (2021), pp.~252--266.

\bibitem{engel2022geometric}
{\sc M.~Engel, F.~Hummel, C.~Kuehn, N.~Popovi{\'c}, M.~Ptashnyk, and
  T.~Zacharis}, {\em Geometric analysis of fast-slow pdes with fold
  singularities}, arXiv preprint arXiv:2207.06134,  (2022).

\bibitem{fenichel1971persistence}
{\sc N.~Fenichel}, {\em Persistence and smoothness of invariant manifolds for
  flows}, Indiana University Mathematics Journal, 21 (1971), pp.~193--226.

\bibitem{fenichel1979geometric}
\leavevmode\vrule height 2pt depth -1.6pt width 23pt, {\em Geometric singular
  perturbation theory for ordinary differential equations}, Journal of
  Differential Equations, 31 (1979), pp.~53--98.

\bibitem{franovic2020dynamics}
{\sc I.~Franovi{\'c}, S.~Yanchuk, S.~Eydam, I.~Ba{\v{c}}i{\'c}, and
  M.~Wolfrum}, {\em Dynamics of a stochastic excitable system with slowly
  adapting feedback}, Chaos: An Interdisciplinary Journal of Nonlinear Science,
  30 (2020).

\bibitem{funaro1991approximation}
{\sc D.~Funaro and O.~Kavian}, {\em Approximation of some diffusion evolution
  equations in unbounded domains by hermite functions}, Mathematics of
  Computation, 57 (1991), pp.~597--619.

\bibitem{gao2016fokker}
{\sc T.~Gao, J.~Duan, and X.~Li}, {\em Fokker--planck equations for stochastic
  dynamical systems with symmetric l{\'e}vy motions}, Applied Mathematics and
  Computation, 278 (2016), pp.~1--20.

\bibitem{gardiner2009stochastic}
{\sc C.~Gardiner}, {\em Stochastic methods}, vol.~4, Springer Berlin
  Heidelberg, 2009.

\bibitem{grabert2006projection}
{\sc H.~Grabert}, {\em The projection operator technique}, Projection Operator
  Techniques in Nonequilibrium Statistical Mechanics,  (2006), pp.~9--26.

\bibitem{guo1999error}
{\sc B.-Y. Guo}, {\em Error estimation of hermite spectral method for nonlinear
  partial differential equations}, Mathematics of Computation, 68 (1999),
  pp.~1067--1078.

\bibitem{hijawi1997nonlinear}
{\sc M.~Hijawi, R.~Ibrahim, and N.~Moshchuk}, {\em Nonlinear random response of
  ocean structures using first-and second-order stochastic averaging},
  Nonlinear Dynamics, 12 (1997), pp.~155--197.

\bibitem{huang2022homogenization}
{\sc Q.~Huang, J.~Duan, and R.~Song}, {\em Homogenization of nonlocal partial
  differential equations related to stochastic differential equations with
  l{\'e}vy noise}, Bernoulli, 28 (2022), pp.~1648--1674.

\bibitem{hummel2022slow}
{\sc F.~Hummel and C.~Kuehn}, {\em Slow manifolds for infinite-dimensional
  evolution equations}, Commentarii Mathematici Helvetici, 97 (2022),
  pp.~61--132.

\bibitem{j2007statistical}
{\sc D.~J~Evans and G.~P~Morriss}, {\em Statistical mechanics of nonequilbrium
  liquids}, ANU Press, 2007.

\bibitem{jones1995geometric}
{\sc C.~Jones}, {\em Geometric singular perturbation theory}, Dynamical
  Systems,  (1995), pp.~44--118.

\bibitem{jones1995gspt}
{\sc C.~K. R.~T. Jones}, {\em Geometric singular perturbation theory}, in
  Dynamical Systems: Lectures Given at the 2nd Session of the Centro
  Internazionale Matematico Estivo (C.I.M.E.) held in Montecatini Terme, Italy,
  June 13--22, 1994, R.~Johnson, ed., Springer Berlin Heidelberg, 1995,
  pp.~44--118.

\bibitem{jordan1998variational}
{\sc R.~Jordan, D.~Kinderlehrer, and F.~Otto}, {\em The variational formulation
  of the fokker--planck equation}, SIAM journal on mathematical analysis, 29
  (1998), pp.~1--17.

\bibitem{kaneko1981adiabatic}
{\sc K.~Kaneko}, {\em Adiabatic elimination by the eigenfunction expansion
  method}, Progress of Theoretical Physics, 66 (1981), pp.~129--142.

\bibitem{knezevic2009spectral}
{\sc D.~J. Knezevic and E.~S{\"u}li}, {\em Spectral galerkin approximation of
  fokker-planck equations with unbounded drift}, ESAIM: Mathematical Modelling
  and Numerical Analysis, 43 (2009), pp.~445--485.

\bibitem{kuehn2015multiple}
{\sc C.~Kuehn}, {\em Multiple time scale dynamics}, vol.~191, Springer, Cham,
  2015.

\bibitem{kuehn2022general}
{\sc C.~Kuehn, N.~Berglund, C.~Bick, M.~Engel, T.~Hurth, A.~Iuorio, and
  C.~Soresina}, {\em A general view on double limits in differential
  equations}, Physica D: Nonlinear Phenomena, 431 (2022), p.~133105.

\bibitem{kuehn2024infinite}
{\sc C.~Kuehn, P.~Lehner, and J.-E. Sulzbach}, {\em Infinite dimensional slow
  manifolds for a linear fast-reaction system}, Contemporary Mathematics, AMS,
  806 (2024), pp.~87--104.

\bibitem{kuehn2023fast}
{\sc C.~Kuehn and J.-E. Sulzbach}, {\em Fast reactions and slow manifolds},
  arXiv preprint arXiv:2301.09368,  (2023).

\bibitem{laing2009stochastic}
{\sc C.~Laing and G.~J. Lord}, {\em Stochastic methods in neuroscience}, OUP
  Oxford, 2009.

\bibitem{le2004renormalized}
{\sc C.~Le~Bris and P.-L. Lions}, {\em Renormalized solutions of some transport
  equations with partially w 1, 1 velocities and applications}, Annali di
  Matematica pura ed applicata, 183 (2004), pp.~97--130.

\bibitem{lindner2004effects}
{\sc B.~Lindner, J.~Garc{\i}a-Ojalvo, A.~Neiman, and L.~Schimansky-Geier}, {\em
  Effects of noise in excitable systems}, Physics reports, 392 (2004),
  pp.~321--424.

\bibitem{lochak2012multiphase}
{\sc P.~Lochak and C.~Meunier}, {\em Multiphase averaging for classical
  systems: with applications to adiabatic theorems}, vol.~72, Springer Science
  \& Business Media, 1988.

\bibitem{lunardi2012analytic}
{\sc A.~Lunardi}, {\em Analytic semigroups and optimal regularity in parabolic
  problems}, Springer Science \& Business Media, 2012.

\bibitem{ma2005hermite}
{\sc H.~Ma, W.~Sun, and T.~Tang}, {\em Hermite spectral methods with a
  time-dependent scaling for parabolic equations in unbounded domains}, SIAM
  Journal on Numerical Analysis, 43 (2005), pp.~58--75.

\bibitem{majda2001mathematical}
{\sc A.~J. Majda, I.~Timofeyev, and E.~Vanden-Eijnden}, {\em A mathematical
  framework for stochastic climate models}, Communications on Pure and Applied
  Mathematics: A Journal Issued by the Courant Institute of Mathematical
  Sciences, 54 (2001), pp.~891--974.

\bibitem{majda2003systematic}
\leavevmode\vrule height 2pt depth -1.6pt width 23pt, {\em Systematic
  strategies for stochastic mode reduction in climate}, Journal of the
  Atmospheric Sciences, 60 (2003), pp.~1705--1722.

\bibitem{manita2016cauchy}
{\sc O.~A. Manita and S.~V. Shaposhnikov}, {\em On the cauchy problem for
  fokker--planck--kolmogorov equations with potential terms on arbitrary
  domains}, Journal of Dynamics and Differential Equations, 28 (2016),
  pp.~493--518.

\bibitem{mattingly2002ergodicity}
{\sc J.~C. Mattingly, A.~M. Stuart, and D.~J. Higham}, {\em Ergodicity for sdes
  and approximations: locally lipschitz vector fields and degenerate noise},
  Stochastic Processes and their Applications, 101 (2002), pp.~185--232.

\bibitem{monahan2011stochastic}
{\sc A.~H. Monahan and J.~Culina}, {\em Stochastic averaging of idealized
  climate models}, Journal of Climate, 24 (2011), pp.~3068--3088.

\bibitem{mori1965transport}
{\sc H.~Mori}, {\em Transport, collective motion, and brownian motion},
  Progress of Theoretical Physics, 33 (1965), pp.~423--455.

\bibitem{pavliotis2008multiscale}
{\sc G.~Pavliotis and A.~Stuart}, {\em Multiscale methods: averaging and
  homogenization}, Springer Science \& Business Media, 2008.

\bibitem{ren2015approximation}
{\sc J.~Ren, J.~Duan, and C.~K.~T. Jones}, {\em Approximation of random slow
  manifolds and settling of inertial particles under uncertainty}, Journal of
  Dynamics and Differential Equations, 27 (2015), pp.~961--979.

\bibitem{risken1996fokker}
{\sc H.~Risken}, {\em The Fokker-Planck Equation}, Springer, 1996.

\bibitem{schertzer2001fractional}
{\sc D.~Schertzer, M.~Larchev{\^e}que, J.~Duan, V.~Yanovsky, and S.~Lovejoy},
  {\em Fractional fokker--planck equation for nonlinear stochastic differential
  equations driven by non-gaussian l{\'e}vy stable noises}, Journal of
  Mathematical Physics, 42 (2001), pp.~200--212.

\bibitem{schmalfuss1998random}
{\sc B.~Schmalfuss}, {\em A random fixed point theorem and the random graph
  transformation}, Journal of Mathematical Analysis and Applications, 225
  (1998), pp.~91--113.

\bibitem{schmalfuss2008invariant}
{\sc B.~Schmalfuss and K.~R. Schneider}, {\em Invariant manifolds for random
  dynamical systems with slow and fast variables}, Journal of Dynamics and
  Differential Equations, 20 (2008), pp.~133--164.

\bibitem{te2020projection}
{\sc M.~te~Vrugt and R.~Wittkowski}, {\em Projection operators in statistical
  mechanics: a pedagogical approach}, European Journal of Physics, 41 (2020),
  p.~045101.

\bibitem{van1985elimination}
{\sc N.~G. Van~Kampen}, {\em Elimination of fast variables}, Physics Reports,
  124 (1985), pp.~69--160.

\bibitem{van1992stochastic}
\leavevmode\vrule height 2pt depth -1.6pt width 23pt, {\em Stochastic processes
  in physics and chemistry}, vol.~1, Elsevier, 1992.

\bibitem{walter2006conditional}
{\sc J.~Walter and C.~Sch{\"u}tte}, {\em Conditional averaging for diffusive
  fast-slow systems: A sketch for derivation}, in Analysis, Modeling and
  Simulation of Multiscale Problems, Springer, 2006, pp.~647--682.

\bibitem{watson1989parabolic}
{\sc N.~A. Watson}, {\em Parabolic equations on an infinite strip}, CRC Press,
  1989.

\bibitem{xu2011averaging}
{\sc Y.~Xu, J.~Duan, and W.~Xu}, {\em An averaging principle for stochastic
  dynamical systems with l{\'e}vy noise}, Physica D: Nonlinear Phenomena, 240
  (2011), pp.~1395--1401.

\bibitem{zolotarev1986one}
{\sc V.~M. Zolotarev}, {\em One-dimensional stable distributions}, vol.~65,
  American Mathematical Soc., 1986.

\bibitem{zwanzig1960ensemble}
{\sc R.~Zwanzig}, {\em Ensemble method in the theory of irreversibility}, The
  Journal of Chemical Physics, 33 (1960), pp.~1338--1341.

\end{thebibliography}
\end{document}